\RequirePackage{fix-cm}
\documentclass[smallextended]{svjour3}       
\smartqed  
\usepackage{graphicx}
\usepackage{etex,bm}
\usepackage{amsmath,amssymb,mathrsfs,mathtools}
\usepackage{amsfonts}
\usepackage{cite}
\usepackage[colorlinks=true,citecolor=blue,linkcolor=blue]{hyperref}
\usepackage[margin=1in]{geometry}
\usepackage{multirow,booktabs}
\usepackage{empheq}
\usepackage{array}
\usepackage{caption}
\usepackage{subcaption}
\usepackage{relsize}
\renewcommand{\figurename}{Figure}

\renewenvironment{proof}{\paragraph{Proof:}}{\hfill$\square$}
\usepackage[utf8]{inputenc}
\begin{document}

\title{Geometric Programming Problems with Triangular and Trapezoidal Two-fold Uncertainty Distributions 
}

\titlerunning{GP problems with triangular and trapezoidal two-fold UDs}        

\author{Tapas Mondal         \and
        Akshay Kumar Ojha  \and Sabyasachi Pani
}


\author{Tapas Mondal* \and Akshay Kumar Ojha \and Sabyasachi Pani }
\institute{Tapas Mondal \at
	School of Basic Sciences, Indian Institute of Technology Bhubaneswar \\
	Tel.: +918345939869\\
	\email{tm19@iitbbs.ac.in}           
	\and
	Akshay Kumar Ojha \at
	School of Basic Sciences, Indian Institute of Technology Bhubaneswar \\
	\email{akojha@iitbbs.ac.in}
	\and
	Sabyasachi Pani \at
	School of Basic Sciences, Indian Institute of Technology Bhubaneswar \\
	\email{spani@iitbbs.ac.in}
}

\date{Received: date / Accepted: date}

\maketitle

\begin{abstract}
Geometric programming (GP) is a well-known optimization tool for dealing with a wide range of nonlinear optimization and engineering problems. In general, it is assumed that the parameters of a GP problem are deterministic and accurate. However, in the real-world GP problem, the parameters are frequently inaccurate and ambiguous. This paper investigates the GP problem in an uncertain environment, with the coefficients as triangular and trapezoidal two-fold uncertain variables. In this paper, we introduce uncertain measures in a generalized version and focus on more complicated two-fold uncertainties to propose triangular and trapezoidal two-fold uncertain variables within the context of uncertainty theory. We develop three reduction methods to convert triangular and trapezoidal two-fold uncertain variables into single-fold uncertain variables using optimistic, pessimistic, and expected value criteria. Reduction methods are used to convert the GP problem with two-fold uncertainty into the GP problem with single-fold uncertainty. Furthermore, the chance-constrained uncertain-based framework is used to solve the reduced single-fold uncertain GP problem. Finally, a numerical example is provided to demonstrate the effectiveness of the procedures.\\
\end{abstract}
\section{Introduction}\label{sec.1}
Geometric programming (GP) is one of the most effective methods for solving nonlinear optimization problems. In 1967, Duffin et al. \cite{Duffin 1967} first introduced the fundamental theories of the GP problem. Generally, the GP technique is used for posynomial types of objective and constraint functions. The GP problem with all positive parameters except the exponents is called the posynomial problem. The GP technique converts the primal problem into a dual problem, which is easier to solve.
\par Nowadays, the GP problem is used in engineering design problems like circuit design \cite{Chu 2001,Hershenson 2001}, inventory modeling \cite{Mandal 2006,Roy 1997,Worrall 1982}, production planning \cite{Cheng 1991,Choi 1996,Islam 2007,Jung 2001,Kim 1998,Lee 1993}, risk management \cite{Scott 1995}, chemical processing \cite{Passy 1968,Ruckaert 1978,Wall 1986}, information theory \cite{Chiang 2005}, and structural design \cite{Gupta 1986}.
\par The traditional GP problem assumes that the parameters are precisely known and deterministic. Several researchers developed efficient and practical algorithms for the conventional GP problem with specific and precise parameters \cite{Avriel 1975,Beightler 1976,Duffin 1967,Duffin 1973,Fang 1988,Kortanek 1992,Kortanek 1997,Maranas 1997,Rajgopal 1990,Rajgopal 1992,Rajgopal 2002}. In practice, however, the values of the parameters of a GP problem may be uncertain and imprecise. As a result, numerous methods to the problem of uncertain GP have been developed. Avriel and Wilde \cite{Avriel 1969} introduced the stochastic GP problem. The authors considered the GP problem, in which the coefficients of objective and constraint functions are nonnegative random variables. Liu \cite{Liu 2006} devised a method to find the range of the posynomial objective values for the GP problem with coefficients and exponents that have interval values.  Mahapatra and Mandal \cite{Mahapatra 2012} solved the GP problem with interval-valued coefficients in a parametric form. 
\par In the last few decades, the GP problem has been developed in fuzzy environments. In 1993, Cao \cite{Cao 1993} extended the GP problem in an imprecise environment, along with interval and fuzzy coefficients. Later on, the same author \cite{Cao 1997} proposed the GP problem with T-fuzzy coefficients. Mandal and Roy \cite{Mandal 2006} solved the GP problem with L-R fuzzy coefficients. Yang and Cao \cite{Yang 2007} made significant contributions to the fuzzy relational GP under monomial.  Liu \cite{Liu 2007} developed the GP problem with fuzzy parameters. Shiraz et al. \cite{Shiraz 2017} used possibility, necessity, and credibility approaches to solve the fuzzy chance-constrained GP problem. Under the rough set theory, Shiraz and Fukuyama \cite{Shiraz 2018} developed the GP problem.
\par Recently, Liu \cite{Liu 2015} pioneered uncertainty theory, a new and developing field of mathematics. In uncertainty theory, a measurable function from uncertainty space to the set of real numbers is known as an uncertain variable (UV). It is characterized in mathematics by its corresponding uncertainty distribution (UD). UVs can be normal, linear, zigzag, log-normal, triangular, or trapezoidal. Using uncertainty theory, many decision-making optimization problems can be solved efficiently and effectively. To solve the GP problem, several researchers developed an uncertainty-based framework based on uncertainty theory. Shiraz et al. \cite{Shiraz 2016} first considered the GP problem in an uncertain environment. Based on uncertainty theory, the authors developed the deterministic form of an uncertain GP problem under normal, linear, and zigzag uncertainties. In 2022, Mondal et al. \cite{Mondal 2022} developed a procedure for solving the GP problem with uncertainty. The authors derived the equivalent deterministic form of the GP problem under triangular and trapezoidal uncertainty.
\par It is noted that the GP problem is only taken into account when there is a single-fold uncertainty. In practice, however, it is observed that uncertainty frequently occurs in multiple aspects of real-world decision-making systems due to a lack of prior information. Thus, two-fold uncertainty is developed in order to describe multi-fold uncertainty. In multi-fold uncertainty, the main aim is to reduce multi-fold UVs into single-fold UVs. In 2015, Yang et al. \cite{Yang 2015} developed the reduction methods of normal, linear, zigzag, and log-normal two-fold UVs. The authors applied those methods in a solid transportation problem.
\par According to the literature review, there has been a lot of research on the crisp and fuzzy GP problem. A GP problem with single-fold uncertainties has also recently emerged. To the best of our knowledge, no previous work on the GP problem with triangular and trapezoidal two-fold uncertainty has been done. As a result, we attempt to investigate the GP problem with triangular and trapezoidal two-fold uncertainty. The following are the main contributions we made to this study.
\begin{description}
	\item[I.] We propose triangular and trapezoidal two-fold UVs within the framework of uncertainty theory by introducing uncertain measures in a generalized version and focusing on more complicated two-fold uncertainties. Moreover, we consider the GP problem in an uncertain environment, with the coefficients as triangular and trapezoidal two-fold UVs. 
	\item[II.]  We develop three reduction methods to reduce triangular and trapezoidal two-fold UVs into single-fold UVs with the help of optimistic, pessimistic, and expected value criteria. 
	\item[III.] Reduction methods are used to convert the GP problem with two-fold uncertainty into the GP problem with single-fold uncertainty. 
	\item[IV.] The chance-constrained uncertain-based framework is used to solve the reduced single-fold uncertain GP problem. 
	\item[V.]  A numerical example is demonstrated to show the efficiency of the procedures.
\end{description}
\par The paper is organized as follows. Triangular and trapezoidal two-fold UVs are proposed in Section \ref{sec.2}. Reduction methods are developed in Section \ref{sec.3}. The GP problem with triangular and trapezoidal two-fold uncertain coefficients are considered and developed the solving procedures in Section \ref{sec.4}. In Section \ref{sec.5}, a numerical example is given. Finally, a conclusion on this work is incorporated in Section \ref{sec.6}.
\section{Triangular and Trapezoidal Two-fold UVs}\label{sec.2} 
In this section, we propose triangular and trapezoidal two-fold UVs. To do so, first, we discuss some preliminary concepts of single-fold uncertainty.
\subsection{Single-fold uncertainty} Uncertainty theory is a new branch of mathematics that uses the concept of UVs to describe the characteristics of single-fold uncertainties in order to model degrees of belief. To define UVs, we must first understand the ideas of uncertain measure and uncertainty space.
\begin{definition}\cite{Liu 2015}\label{def.1}
Let $L$ be the $\sigma-$algebra over the nonempty set $\Gamma$. In order to define UV, we need to first define the uncertain measure and uncertainty space. The uncertain measure is the set function $\mathcal{M}:L\rightarrow[0,1],$ which has the following axioms.
\begin{description}
	\item[I.] (Normality Axiom) $\mathcal{M}\{\Gamma\}=1$.
	\item[II.] (Duality Axiom) $\mathcal{M}\{A\}+\mathcal{M}\{A^c\}=1$ for any event $A\in L$.
	\item[III.] (Subadditivity Axiom) $\mathcal{M}\big\{ \bigcup\limits_{i=1}^{\infty}A_i \big\} \leq \sum\limits_{i=1}^{\infty}\mathcal{M}\{A_i\}$ for every countable sequence of events $\{A_i\}_{i=1}^{\infty}$.
\end{description}
\par Moreover, the triplet $(\Gamma,L,\mathcal{M})$ is known as the uncertainty space.
\end{definition}
\begin{remark}\cite{Liu 2015}
	Uncertainty is defined as one's personal level of belief that an event will occur. It is determined by personal knowledge of the event. If the state of knowledge changes, the uncertain measure will change.
	\end{remark}
\par The concept of UV is extremely useful for describing uncertainties. Here is the definition of UV.
\begin{definition}\cite{Liu 2015}\label{def.2}
	 A measurable function $\xi:(\Gamma,L,\mathcal{M})\to \mathbb{R}$ is called a UV, where $(\Gamma,L,\mathcal{M})$ is the uncertainty space and $\mathbb{R}$ is the set of real numbers, i.e., for any Borel set $B$, $\{\xi\in B\}$ is an event.
\end{definition}
\par Alternatively, a UV can be clearly characterized by its corresponding UD, which is a carrier of uncertain information about the UV. In most cases, knowing the UD rather than the UV itself is sufficient. Here is the mathematical definition of UD.
\begin{definition}\cite{Liu 2015}\label{def.3}
 If $\Phi$ is the UD corresponding to a UV $\xi$, then $\Phi$ is defined as 
	\begin{equation*}
		\Phi(x)=\mathcal{M}\{\xi\le x\},\forall x\in \mathbb{R}.
	\end{equation*}
\end{definition}
\par Moreover, the necessary and sufficient condition for a UD function is stated as follows.
\begin{theorem}\cite{Liu 2015}\label{thm.1}
	A function $\Phi:\mathbb{R}\to[0,1]$ is a UD if and only if it is a monotonic increasing function except $\Phi(x)=0$ and $\Phi(x)=1$.
	\end{theorem}
\par It is simple to describe uncertain information geometrically using the UD in a two-dimensional space. Several UVs, like normal, linear, zigzag, log-normal, triangular, and trapezoidal, are used in real-world decision-making systems. Linear UV has an important role in generalized uncertainty. Here is the definition of linear UV.
\begin{definition}\cite{Liu 2015}\label{def.4}
	Let $\Phi$ be the UD corresponding to a UV $\xi$. The UV $\xi$ is said to be linear if and only if $\Phi$ is defined as follows.\\
	\[\Phi(x)=\left\{
	\begin{array}{ll}
		0, & x\leq a;\\
		\frac{x-a}{b-a}, & a\leq x\leq b;\\
		1, & x \geq b;\\
	\end{array}
	\right.
	\]
	where $a,b \in \mathbb{R}, a<b.$
	If $\xi$ is a linear UV with the parameters $a$ and $b$, then we write it as $\xi\sim \mathcal{L}(a,b).$
\end{definition}

\par In this research, our main work on triangular and trapezoidal uncertainty. In 2022, Mondal et al. \cite{Mondal 2022} proposed triangular and trapezoidal single-fold UVs. Here are the definitions of those UVs.
\begin{definition}\cite{Mondal 2022}\label{def.5}
	Let $\Phi$ be the UD corresponding to a UV $\xi$. The UV $\xi$ is said to be triangular if and only if $\Phi$ is defined as follows.\\
	\[\Phi(x)=\left\{
	\begin{array}{ll}
		0, & x\leq a;\\
		\frac{(x-a)^2}{(b-a)(c-a)}, & a\leq x\leq b;\\
		1-\frac{(c-x)^2}{(c-a)(c-b)}, & b\leq x\leq c;\\
		1, & x \geq c;\\
	\end{array}
	\right.
	\]
	where $a,b,c \in \mathbb{R}, a<b<c.$
	If $\xi$ is a triangular UV with the parameters $a$, $b$, and $c$, then we write it as $\xi\sim \mathcal{TRI}(a,b,c).$
\end{definition} 
	\begin{definition}\cite{Mondal 2022}\label{def.6}
	Let $\Phi$ be the UD corresponding to a UV $\xi$. The UV $\xi$ is said to be trapezoidal if and only if $\Phi$ is defined as follows.\\
	\[\Phi(x)=\left\{
	\begin{array}{ll}
		0, & x\leq a;\\
		\frac{(x-a)^2}{(d+c-a-b)(b-a)}, & a\leq x\leq b;\\
		\frac{(2x-a-b)}{(d+c-a-b)}, & b\leq x\leq c;\\
		1-\frac{(d-x)^2 }{(d+c-a-b)(d-c)}, & c\leq x\leq d;\\
		1, & x \geq d;\\
	\end{array}
	\right.
	\]
	where $a,b,c,d \in \mathbb{R}, a<b<c<d.$
	If $\xi$ is a trapezoidal UV with the parameters $a$, $b$, $c$, and $d$, then we write it as $\xi\sim \mathcal{TRA}(a,b,c,d).$
\end{definition} 
\subsection{Two-fold uncertainty} In a real-world situation, there may be inconsistencies in various experts' information on an event. Due to this lack of consistency, it is hard to give a good estimate of how uncertain an event is. In this case, it makes sense to generalize the uncertain measure by using UV values in the interval $[0, 1]$ rather than constant values. The following discussion gives the concepts of generalized uncertain measure, generalized uncertainty space, and two-fold UVs. To begin, we must define a regular UV.
\begin{definition}\cite{Liu 2015}\label{def.7}
	Let $\Phi$ be the UD corresponding to a UV $\xi$.
	The UV $\xi$ is said to be regular if and only if $\Phi$ is continuous and strictly increasing with respect $x$ at which $0<\Phi(x)<1$, and $$\lim\limits_{x\to -\infty}\Phi(x)=0, \lim\limits_{x\to \infty}\Phi(x)=1.$$
	\par Alternatively, a UV $\xi$ is said to be regular if and only if the inverse $\Phi^{-1}(\alpha)$ of the UD $\Phi(x)$ corresponding to the UV $\xi$ exists and it is unique for each $\alpha \in (0,1)$.
\end{definition} 
Similar to Definition \ref{def.1}, we give the concept of generalized uncertain measure.
\begin{definition}\cite{Yang 2015}\label{def.8}
Let $L$ be the $\sigma-$algebra over the nonempty set $\Gamma$. To define two-fold UV, we need to first define the generalized uncertain measure and generalized uncertainty space. The generalized uncertain measure is the set function $\mathcal{\tilde{M}}:L\rightarrow[0,1],$ which has the following two axioms.
\begin{description}
	\item[ Axiom 1.] $\mathcal{\tilde{M}}\{\phi\}=0,\mathcal{\tilde{M}}\{\Gamma\}=1$.
	\item[Axiom 2.] $\mathcal{\tilde{M}}\{A\}+\mathcal{\tilde{M}}\{A^c\}\leq1$ for any event $A\in L$.
\end{description}
\par Moreover, the triplet $(\Gamma,L,\mathcal{\tilde{M}})$ is known as the generalized uncertainty space.
\end{definition} 
\begin{remark}\cite{Yang 2015}
	The set function $\mathcal{\tilde{M}}$ is said to be a generalized uncertain measure if $\mathcal{\tilde{M}}\{A\}$ is a regular UV for any measurable set $A\in L$.
	\end{remark}
\par To describe two-fold uncertainties, the concept of two-fold UV is very useful.
\begin{definition}\cite{Yang 2015}\label{def.9}
	A two-fold UV $\tilde{\xi}$ is a measurable function from the generalized uncertainty space $(\Gamma,L,{\mathcal{\tilde{M}}})$ to the set of real numbers $\mathbb{R}$ such that $\{\eta\in \Gamma|\tilde{\xi}(\eta)\in B\}\in L$ for any Borel set $B$.
\end{definition}
\par A two-fold UV can also be defined by its UD function. Because we use regular UVs to quantify the likelihood of uncertain events, the corresponding distribution function is referred to as the two-fold UD function in the discussion that follows.
\begin{definition}\cite{Yang 2015}\label{def.10}
	 The mapping $\tilde{\Phi}_{\tilde{\xi}}:\mathbb{R}\times S_x \to [0,1]$ is called a two-fold UD. If $\tilde{\Phi}_{\tilde{\xi}}$ is a two-fold UD corresponding to a two-fold UV $\tilde{\xi}$, then $\tilde{\Phi}_{\tilde{\xi}}$ is defined as $$\tilde{\Phi}_{\tilde{\xi}}(x,y)=\mathcal{\tilde{M}}\{\Phi_{\tilde{\xi}}(x)\leq y\}, (x,y)\in \mathbb{R}\times S_x,$$
	where $\Phi_{\tilde{\xi}}(x)=\mathcal{M}\{\eta\in \Gamma|\tilde{\xi}(\eta)\leq x\}$, $S_x\subset [0,1]$ is the support of $\Phi_{\tilde{\xi}}(x)$, i.e., $$S_x=\{y\in[0,1]\big|0<\tilde{\Phi}_{\tilde{\xi}}(x,y)<1\}.$$
\end{definition}
\par In this research, we propose triangular and trapezoidal two-fold UVs. The proposed definitions are given as follows.
\begin{definition}\label{def.11}
	Let  $\tilde{\xi}$ be a triangular two-fold UV, denoted by $\mathcal{TRI}(a,b,c;\theta_l,\theta_r)$. The two-fold UD $\tilde{\Phi}_{\tilde{\xi}}(x,y)$ of a triangular two-fold UV $\tilde{\xi}$ is defined as follows.
	\begin{description}
		\item[I.] If $x \leq a$, then $\tilde{\Phi}_{\tilde{\xi}}(x,y)=0$.
		\item[II.] If $a<x<b$, then
	\[\tilde{\Phi}_{\tilde{\xi}}(x,y)=\left\{
	\begin{array}{ll}
		0, & y\leq A;\\
		\frac{y-A}{B-A}, & A\leq y\leq B;\\
		1, & y \geq B;\\
	\end{array}
	\right.
	\]
	where $$A=\frac{(x-a)^2}{(b-a)(c-a)}-\theta_l\min\bigg\{\frac{(x-a)^2}{(b-a)(c-a)},\frac{b-a}{c-a}-\frac{(x-a)^2}{(b-a)(c-a)}\bigg\},$$  $$B=\frac{(x-a)^2}{(b-a)(c-a)}+\theta_r\min\bigg\{\frac{(x-a)^2}{(b-a)(c-a)},\frac{b-a}{c-a}-\frac{(x-a)^2}{(b-a)(c-a)}\bigg\}.$$
	\item[III.] If $x =b$, then $\tilde{\Phi}_{\tilde{\xi}}(x,y)=\frac{b-a}{c-a}$.
	\item[IV.] If $b<x<c$, then
	\[\tilde{\Phi}_{\tilde{\xi}}(x,y)=\left\{
	\begin{array}{ll}
		0, & y\leq C;\\
		\frac{y-C}{D-C}, & C\leq y\leq D;\\
		1, & y \geq D;\\
	\end{array}
	\right.
	\]
	where $$C=1-\frac{(c-x)^2}{(c-a)(c-b)}-\theta_l\min\bigg\{\frac{c-b}{c-a}-\frac{(c-x)^2}{(c-a)(c-b)}, \frac{(c-x)^2}{(c-a)(c-b)}\bigg\},$$  $$D=1-\frac{(c-x)^2}{(c-a)(c-b)}+\theta_r\min\bigg\{\frac{c-b}{c-a}-\frac{(c-x)^2}{(c-a)(c-b)}, \frac{(c-x)^2}{(c-a)(c-b)}\bigg\}.$$
	\item[V.] If $x \geq c$, then $\tilde{\Phi}_{\tilde{\xi}}(x,y)=1$.
\end{description}
\par In this definition, $\theta_l, \theta_r \in [0,1]$ are two parameters that characterize the degree of uncertainty of $\tilde{\xi}$ when $x$ is given.
\end{definition}
\begin{remark}
	Alternatively, the two-fold UD of a triangular two-fold UV $\tilde{\xi}$ can be presented as 
		\[\tilde{\Phi}_{\tilde{\xi}}(x,y)=\left\{
	\begin{array}{ll}
		0, & x\leq a;\\
		\mathcal{L}(A,B), & a< x< b;\\
		\frac{b-a}{c-a}, & x=b;\\
		\mathcal{L}(C,D), & b< x< c;\\
		1, & x\geq c;\\
	\end{array}
	\right.
	\]
	where $A,B,C,D$ are given in Definition \ref{def.11}.
	\end{remark}
\begin{definition}\label{def.12}
	Let  $\tilde{\xi}$ be a trapezoidal two-fold UV, denoted by $\mathcal{TRA}(a,b,c,d;\theta_l,\theta_r)$. The two-fold UD $\tilde{\Phi}_{\tilde{\xi}}(x,y)$ of a trapezoidal two-fold UV $\tilde{\xi}$ is defined as follows.
	\begin{description}
		\item[I.] If $x \leq a$, then $\tilde{\Phi}_{\tilde{\xi}}(x,y)=0$.
		\item[II.] If $a<x<b$, then
		\[\tilde{\Phi}_{\tilde{\xi}}(x,y)=\left\{
		\begin{array}{ll}
			0, & y\leq A;\\
			\frac{y-A}{B-A}, & A\leq y\leq B;\\
			1, & y \geq B;\\
		\end{array}
		\right.
		\]
		where $$A=\frac{(x-a)^2}{(d+c-a-b)(b-a)}-\theta_l\min\bigg\{\frac{(x-a)^2}{(d+c-a-b)(b-a)},\frac{b-a}{d+c-a-b}-\frac{(x-a)^2}{(d+c-a-b)(b-a)}\bigg\},$$  $$B=\frac{(x-a)^2}{(d+c-a-b)(b-a)}+\theta_r\min\bigg\{\frac{(x-a)^2}{(d+c-a-b)(b-a)},\frac{b-a}{d+c-a-b}-\frac{(x-a)^2}{(d+c-a-b)(b-a)}\bigg\}.$$
		\item[III.] If $x =b$, then $\tilde{\Phi}_{\tilde{\xi}}(x,y)=\frac{b-a}{d+c-a-b}$.
		\item[IV.] If $b<x<c$, then
		\[\tilde{\Phi}_{\tilde{\xi}}(x,y)=\left\{
		\begin{array}{ll}
			0, & y\leq C;\\
			\frac{y-C}{D-C}, & C\leq y\leq D;\\
			1, & y \geq D;\\
		\end{array}
		\right.
		\]
		where $$C=\frac{2x-a-b}{d+c-a-b}-\theta_l\min\bigg\{\frac{2x-a-b}{d+c-a-b}-\frac{b-a}{d+c-a-b},\frac{2c-a-b}{d+c-a-b}-\frac{2x-a-b}{d+c-a-b}\bigg\},$$  $$D=\frac{2x-a-b}{d+c-a-b}+\theta_r\min\bigg\{\frac{2x-a-b}{d+c-a-b}-\frac{b-a}{d+c-a-b},\frac{2c-a-b}{d+c-a-b}-\frac{2x-a-b}{d+c-a-b}\bigg\}.$$
		\item[V.] If $x= c$, then $\tilde{\Phi}_{\tilde{\xi}}(x,y)=\frac{2c-a-b}{d+c-a-b}$.
		\item[VI.] If $c<x<d$, then
		\[\tilde{\Phi}_{\tilde{\xi}}(x,y)=\left\{
		\begin{array}{ll}
			0, & y\leq E;\\
			\frac{y-E}{F-E}, & E\leq y\leq F;\\
			1, & y \geq F;\\
		\end{array}
		\right.
		\]
		where $$E=1-\frac{(d-x)^2}{(d+c-a-b)(d-c)}-\theta_l\min\bigg\{\frac{d-c}{d+c-a-b}-\frac{(d-x)^2}{(d+c-a-b)(d-c)},\frac{(d-x)^2}{(d+c-a-b)(d-c)}\bigg\},$$  $$F=1-\frac{(d-x)^2}{(d+c-a-b)(d-c)}+\theta_r\min\bigg\{\frac{d-c}{d+c-a-b}-\frac{(d-x)^2}{(d+c-a-b)(d-c)},\frac{(d-x)^2}{(d+c-a-b)(d-c)}\bigg\}.$$
		\item[VII.] If $x \geq d$, then $\tilde{\Phi}_{\tilde{\xi}}(x,y)=1$.
	\end{description}
	\par In this definition, $\theta_l, \theta_r \in [0,1]$ are two parameters that characterize the degree of uncertainty of $\tilde{\xi}$ when $x$ is given.
\end{definition}
\begin{remark}
	Similarly, the two-fold UD of a trapezoidal two-fold UV $\tilde{\xi}$ can be presented as 
	\[\tilde{\Phi}_{\tilde{\xi}}(x,y)=\left\{
	\begin{array}{ll}
		0, & x\leq a;\\
		\mathcal{L}(A,B), & a< x< b;\\
		\frac{b-a}{d+c-a-b}, & x=b;\\
		\mathcal{L}(C,D), & b< x< c;\\
		\frac{2c-a-b}{d+c-a-b}, & x=c;\\
		\mathcal{L}(E,F), & c< x< d;\\
		1, & x\geq d;\\
	\end{array}
	\right.
	\]
	where $A,B,C,D,E,F$ are given in Definition \ref{def.12}.
\end{remark}
\section{ Reduction Methods}\label{sec.3} 
	In this section, we show how to reduce two-fold UV to single-fold UV. To accomplish this, we employ three critical values criteria: optimistic, pessimistic, and expected value criteria. To that end, we will first introduce the concepts of critical values of single-fold UV. Furthermore, the critical values of triangular and trapezoidal single-fold UVs are estimated.
	\subsection{Critical values of single-fold UVs} Liu \cite{Liu 2015} proposed the concepts critical values for single-fold UVs as follows.
	\begin{definition}\cite{Liu 2015}\label{def.13}
		If $\xi$ is a UV, then the optimistic, pessimistic, and expected values of $\xi$ are defined as follows.
		\begin{description}
			\item[ I.]  $\xi_{\sup}(\alpha)=\sup\{r|\mathcal{M}\{\xi\geq r\}\geq \alpha\},\alpha\in (0,1).$
			\item[ II.]  $\xi_{\inf}(\alpha)=\inf\{r|\mathcal{M}\{\xi\leq r\}\geq \alpha\},\alpha\in (0,1).$
			\item[ III.]  $\xi_{exp}=\int\limits_{0}^{\infty}\mathcal{M}\{\xi\geq r\}dr-
			\int\limits_{-\infty}^{0}\mathcal{M}\{\xi\leq r\}dr.$
		\end{description}
	\end{definition} 
\par Alternatively, Definition \ref{def.13} can be represented in terms of UD. The following theorem, known as the measure inversion theorem, is extremely useful in this regard.
\begin{theorem}\cite{Liu 2015}\label{thm.2}
Let $\xi$ be a regular UV with UD $\Phi$. Then for any real number $r$, $\mathcal{M}\{\xi\leq r\}=\Phi(r)\text{ and } \mathcal{M}\{\xi\geq r\}=1-\Phi(r).$
\end{theorem}
\par Based on Theorem \ref{thm.2}, Definition \ref{def.13} can be redefined as
\begin{description}
	\item[ I.]  $\xi_{\sup}(\alpha)=\sup\{r|1-\Phi(r)\geq \alpha\},\alpha\in (0,1).$
	\item[ II.]  $\xi_{\inf}(\alpha)=\inf\{r|\Phi(r)\geq \alpha\},\alpha\in (0,1).$
	\item[ III.]  $\xi_{exp}=\int\limits_{0}^{\infty}(1-\Phi(r))dr-
	\int\limits_{-\infty}^{0}\Phi(r)dr.$
\end{description}
\begin{remark}
	Integrating by parts, we have 
	$\int\limits_{0}^{\infty}(1-\Phi(r))dr-
	\int\limits_{-\infty}^{0}\Phi(r)dr=\int\limits_{-\infty}^{\infty}r{\Phi}'(r)dr.$ Therefore, the expected value of a UV $\xi$ can be determined by the formula given as $\xi_{exp}=\int\limits_{-\infty}^{\infty}r{\Phi}'(r)dr.$ 
	\end{remark}
\par Based on the above operations, the following useful theorem is developed.
\begin{theorem}\cite{Liu 2015}\label{thm.3}
	Let $\xi\sim \mathcal{L}(a,b)$ be a linear UV. The optimistic, pessimistic, and expected values of $\xi$ are as follows.
	\begin{description}
		\item[ I.]  $\xi_{\sup}(\alpha)=\alpha a+(1-\alpha)b,\alpha\in (0,1).$
		\item[ II.]  $\xi_{\inf}(\alpha)=(1-\alpha)a+\alpha b,\alpha\in (0,1).$
		\item[ III.]  $\xi_{exp}=\frac{a+b}{2}.$
	\end{description} 
\end{theorem}
\par Similarly, we developed the critical values of triangular and trapezoidal UVs.
\begin{theorem}\label{thm.4}
	Let $\xi\sim \mathcal{TRI}(a,b,c)$ be a triangular UV. The optimistic, pessimistic, and expected values of $\xi$ are as follows.
	\begin{description}
		\item[ I.]  
		\[\xi_{\sup}(\alpha)=\left\{
		\begin{array}{ll}
			a+\sqrt{(1-\alpha)(b-a)(c-a)}, & 0<\alpha\leq \frac{b-a}{c-a};\\
			c-\sqrt{\alpha(c-a)(c-b)}, &  \frac{b-a}{c-a}\leq\alpha<1.\\
		\end{array}
		\right.
		\]
		\item[ II.]  \[\xi_{\inf}(\alpha)=\left\{
		\begin{array}{ll}
			a+\sqrt{\alpha(b-a)(c-a)}, & 0<\alpha\leq \frac{b-a}{c-a};\\
			c-\sqrt{(1-\alpha)(c-a)(c-b)}, &  \frac{b-a}{c-a}\leq\alpha<1.\\
		\end{array}
		\right.
		\]
		\item[ III.]  $$\xi_{exp}=\frac{a+b+c}{3}.$$
	\end{description} 
\end{theorem}
\begin{proof}
	We prove case (I) and case (III). The proof of case (II) is similar to the proof of case (I).\\
	For any $0<\alpha\leq \frac{b-a}{c-a},$ we have
	\begin{equation*} 
	\begin{aligned}
	\xi_{\sup}(\alpha)&=\sup\{r|1-\Phi(r)\geq \alpha\}\\
&=\sup\bigg\{r\bigg|1-\frac{(r-a)^2}{(b-a)(c-a)}\geq \alpha\bigg\}\\
&=\sup\bigg\{r\bigg|\frac{(r-a)^2}{(b-a)(c-a)}\leq 1-\alpha\bigg\}\\
&=\sup\{r|(r-a)^2\leq (1-\alpha)(b-a)(c-a)\}\\
&=\sup\bigg\{r|-\sqrt{(1-\alpha)(b-a)(c-a)}\leq (r-a)\leq \sqrt{(1-\alpha)(b-a)(c-a)}\bigg\}\\
&=\sup\bigg\{r|a-\sqrt{(1-\alpha)(b-a)(c-a)}\leq r\leq a+\sqrt{(1-\alpha)(b-a)(c-a)}\bigg\}\\
&=a+\sqrt{(1-\alpha)(b-a)(c-a)}.
\end{aligned}
\end{equation*}
For any $\frac{b-a}{c-a}\leq\alpha<1,$ we have
\begin{equation*} 
	\begin{aligned}
		\xi_{\sup}(\alpha)&=\sup\{r|1-\Phi(r)\geq \alpha\}\\
		&=\sup\bigg\{r\bigg|1-\bigg[1-\frac{(c-r)^2}{(c-a)(c-b)}\bigg]\geq \alpha\bigg\}\\
		&=\sup\bigg\{r\bigg|\frac{(c-r)^2}{(c-a)(c-b)}\geq \alpha\bigg\}\\
		&=\sup\{r|(c-r)^2\geq \alpha (c-a)(c-b)\}
	\end{aligned}
\end{equation*}
If $c-r\leq -\sqrt{\alpha(c-a)(c-b)},$ then $r\geq c+\sqrt{\alpha(c-a)(c-b)}.$ So, $\Phi(r)=1,$ which implies $1-\Phi(r)=0\geq \alpha.$ This contradicts $\alpha \in (0,1).$
Thus, we omit the case $c-r\leq -\sqrt{\alpha(c-a)(c-b)}.$  Therefore, we have 
\begin{equation*}
	\begin{aligned}
		\xi_{\sup}(\alpha)&=\sup\bigg\{r| c-r\geq \sqrt{\alpha(c-a)(c-b)}\bigg\}\\
		&=\sup\bigg\{r|r\leq c-\sqrt{\alpha(c-a)(c-b)}\bigg\}\\
		&=c-\sqrt{\alpha(c-a)(c-b)}.
	\end{aligned}
\end{equation*}
Hence the proof of case (I) is done.
\par Next, we proof the case (III). We have 
\begin{equation*} 
	\begin{aligned}
		\xi_{exp}&=\int\limits_{-\infty}^{\infty}r{\Phi}'(r)dr\\
		&=\int\limits_{a}^{b}r\cdot \frac{2(r-a)}{(b-a)(c-a)}dr+\int\limits_{b}^{c}r\cdot \frac{2(c-r)}{(c-a)(c-b)}dr\\
		&=\frac{2b^2-a^2-ab}{3(c-a)}+\frac{c^2+bc-2b^2}{3(c-a)}\\
		&=\frac{a+b+c}{3}.
	\end{aligned}
\end{equation*}
Hence the proof of case (III) is done.
\end{proof}
\begin{theorem}\label{thm.5}
	Let $\xi\sim \mathcal{TRA}(a,b,c,d)$ be a trapezoidal UV. The optimistic, pessimistic, and expected values of $\xi$ are as follows.
	\begin{description}
		\item[ I.]  
		\[\xi_{\sup}(\alpha)=\left\{
		\begin{array}{ll}
			a+\sqrt{(1-\alpha)(d+c-a-b)(b-a)}, & 0<\alpha\leq \frac{b-a}{d+c-a-b};\\
			\frac{(a+b)+(1-\alpha)(d+c-a-b)}{2}, &  \frac{b-a}{d+c-a-b}\leq\alpha\leq \frac{2c-a-b}{d+c-a-b} ;\\
			d-\sqrt{\alpha(d+c-a-b)(d-c)}, &  \frac{2c-a-b}{d+c-a-b}\leq\alpha<1.\\
		\end{array}
		\right.
		\]
		\item[ II.]  \[\xi_{\inf}(\alpha)=\left\{
		\begin{array}{ll}
			a+\sqrt{\alpha(d+c-a-b)(b-a)}, & 0<\alpha\leq \frac{b-a}{d+c-a-b};\\
			\frac{(a+b)+\alpha(d+c-a-b)}{2}, &  \frac{b-a}{d+c-a-b}\leq\alpha\leq \frac{2c-a-b}{d+c-a-b} ;\\
			d-\sqrt{(1-\alpha)(d+c-a-b)(d-c)}, &  \frac{2c-a-b}{d+c-a-b}\leq\alpha<1.\\
		\end{array}
		\right.
		\]
		\item[ III.]  $$\xi_{exp}=\frac{1}{3(d+c-a-b)}\bigg[\frac{d^3-c^3}{d-c}-\frac{b^3-a^3}{b-a}\bigg].$$
	\end{description} 
\end{theorem}
\begin{proof}
	The proof is similar to the proof of Theorem \ref{thm.4}.
\end{proof}
\subsection{Reduction methods of two-fold UVs with critical values criteria}
Using optimistic, pessimistic, and expected value criteria, we develop three different reduction methods to convert triangular two-fold UV into single-fold UV. The following theorems are developed in this regard.
\begin{theorem}\label{thm.6}
	Let $\tilde{\xi}\sim \mathcal{TRI}(a,b,c;\theta_l,\theta_r)$ be a triangular two-fold UV. The reduced single-fold UD via $\alpha$ optimistic value criteria for any $\alpha \in(0,1)$ is as follows.
	\begin{description}
		\item[ I.] If $x\leq a$, then  $\Phi_{\tilde{\xi}_{\sup}}(x;\alpha)=0.$  
		\item[ II.] If $a< x< b$, then
		\[\Phi_{\tilde{\xi}_{\sup}}(x;\alpha)=\left\{
		\begin{array}{ll}
			\frac{(x-a)^2}{(b-a)(c-a)}-[\alpha\theta_l-(1-\alpha)\theta_r]\frac{(x-a)^2}{(b-a)(c-a)}, & a<x\leq a+\frac{b-a}{\sqrt{2}};\\
			\frac{(x-a)^2}{(b-a)(c-a)}-[\alpha\theta_l-(1-\alpha)\theta_r]\big[\frac{b-a}{c-a}-\frac{(x-a)^2}{(b-a)(c-a)}\big], & a+\frac{b-a}{\sqrt{2}}\leq x<b.\\
		\end{array}
		\right.
		\]
		\item[ III.] If $x=b$, then $\Phi_{\tilde{\xi}_{\sup}}(x;\alpha)=\frac{b-a}{c-a}.$ 
		\item[ IV.] If $b< x< c$, then
		\[\Phi_{\tilde{\xi}_{\sup}}(x;\alpha)=\left\{
		\begin{array}{ll}
			1-\frac{(c-x)^2}{(c-a)(c-b)}-[\alpha\theta_l-(1-\alpha)\theta_r]\big[\frac{c-b}{c-a}-\frac{(c-x)^2}{(c-a)(c-b)}\big], & b<x\leq c-\frac{c-b}{\sqrt{2}};\\
			1-\frac{(c-x)^2}{(c-a)(c-b)}-[\alpha\theta_l-(1-\alpha)\theta_r]\frac{(c-x)^2}{(c-a)(c-b)}, &  c-\frac{c-b}{\sqrt{2}}\leq x<c.\\
		\end{array}
		\right.
		\] 
		\item[ V.] If $x\geq c$, then $\Phi_{\tilde{\xi}_{\sup}}(x;\alpha)=1.$  
	\end{description}
\end{theorem}
\begin{proof}
	Let $\alpha\in (0,1)$. 
	
	\par If $x\leq a,$ then we have $\tilde{\Phi}_{\tilde{\xi}}(x,y)=0.$ Therefore, $\Phi_{\tilde{\xi}_{\sup}}(x;\alpha)=0,\forall x\leq a.$
	\par If $a< x< b$, then we have $\tilde{\Phi}_{\tilde{\xi}}(x,y)=\mathcal{L}(A,B),$ where $A,B$ are given in Definition \ref{def.11}. So, based on Theorem \ref{thm.3}, we have 
	\begin{equation*} 
		\begin{aligned}
			\Phi_{\tilde{\xi}_{\sup}}(x;\alpha)&=\alpha A+(1-\alpha)B\\
			&=\alpha \bigg[\frac{(x-a)^2}{(b-a)(c-a)}-\theta_l\min\bigg\{\frac{(x-a)^2}{(b-a)(c-a)},\frac{b-a}{c-a}-\frac{(x-a)^2}{(b-a)(c-a)}\bigg\}\bigg]\\
			&\qquad+(1-\alpha)\bigg[\frac{(x-a)^2}{(b-a)(c-a)}+\theta_r\min\bigg\{\frac{(x-a)^2}{(b-a)(c-a)},\frac{b-a}{c-a}-\frac{(x-a)^2}{(b-a)(c-a)}\bigg\}\bigg]\\
			& =\left\{
			\begin{array}{ll}
				\frac{(x-a)^2}{(b-a)(c-a)}-[\alpha\theta_l-(1-\alpha)\theta_r]\frac{(x-a)^2}{(b-a)(c-a)}, & a<x\leq a+\frac{b-a}{\sqrt{2}};\\
				\frac{(x-a)^2}{(b-a)(c-a)}-[\alpha\theta_l-(1-\alpha)\theta_r]\big[\frac{b-a}{c-a}-\frac{(x-a)^2}{(b-a)(c-a)}\big], & a+\frac{b-a}{\sqrt{2}}\leq x<b.\\
			\end{array}
			\right.
		\end{aligned}
	\end{equation*}
\par If $x=b,$ then we have $\tilde{\Phi}_{\tilde{\xi}}(x,y)=\frac{b-a}{c-a}.$ Therefore, $\Phi_{\tilde{\xi}_{\sup}}(x;\alpha)=\frac{b-a}{c-a}$ for $ x=b.$
\par If $b< x< c$, then we have $\tilde{\Phi}_{\tilde{\xi}}(x,y)=\mathcal{L}(C,D),$ where $C,D$ are given in Definition \ref{def.11}. So, based on Theorem \ref{thm.3}, we have 
\begin{equation*} 
	\begin{aligned}
		\Phi_{\tilde{\xi}_{\sup}}(x;\alpha)&=\alpha C+(1-\alpha)D\\
		&=\alpha \bigg[1-\frac{(c-x)^2}{(c-a)(c-b)}-\theta_l\min\bigg\{\frac{c-b}{c-a}-\frac{(c-x)^2}{(c-a)(c-b)}, \frac{(c-x)^2}{(c-a)(c-b)}\bigg\}\bigg]\\
		&\qquad+(1-\alpha)\bigg[1-\frac{(c-x)^2}{(c-a)(c-b)}+\theta_r\min\bigg\{\frac{c-b}{c-a}-\frac{(c-x)^2}{(c-a)(c-b)}, \frac{(c-x)^2}{(c-a)(c-b)}\bigg\}\bigg]\\
		& =\left\{
		\begin{array}{ll}
			1-\frac{(c-x)^2}{(c-a)(c-b)}-[\alpha\theta_l-(1-\alpha)\theta_r]\big[\frac{c-b}{c-a}-\frac{(c-x)^2}{(c-a)(c-b)}\big], & b<x\leq c-\frac{c-b}{\sqrt{2}};\\
			1-\frac{(c-x)^2}{(c-a)(c-b)}-[\alpha\theta_l-(1-\alpha)\theta_r]\frac{(c-x)^2}{(c-a)(c-b)}, &  c-\frac{c-b}{\sqrt{2}}\leq x<c.\\
		\end{array}
		\right.
	\end{aligned}
\end{equation*}
\par 	If $x\geq c,$ then we have $\tilde{\Phi}_{\tilde{\xi}}(x,y)=1.$ Therefore, $\Phi_{\tilde{\xi}_{\sup}}(x;\alpha)=1, \forall x\geq c.$
\end{proof}
\begin{theorem}\label{thm.7}
	Let $\tilde{\xi}\sim \mathcal{TRI}(a,b,c;\theta_l,\theta_r)$ be a triangular two-fold UV. The reduced single-fold UD via $\alpha$ pessimistic value criteria for any $\alpha \in(0,1)$ is as follows.
	\begin{description}
		\item[ I.] If $x\leq a$, then $\Phi_{\tilde{\xi}_{\inf}}(x;\alpha)=0.$  
		\item[ II.] If $a< x< b$, then
		\[\Phi_{\tilde{\xi}_{\inf}}(x;\alpha)=\left\{
		\begin{array}{ll}
			\frac{(x-a)^2}{(b-a)(c-a)}-[(1-\alpha)\theta_l-\alpha\theta_r]\frac{(x-a)^2}{(b-a)(c-a)}, & a<x\leq a+\frac{b-a}{\sqrt{2}};\\
			\frac{(x-a)^2}{(b-a)(c-a)}-[(1-\alpha)\theta_l-\alpha\theta_r]\big[\frac{b-a}{c-a}-\frac{(x-a)^2}{(b-a)(c-a)}\big], & a+\frac{b-a}{\sqrt{2}}\leq x<b.\\
		\end{array}
		\right.
		\]
		\item[ III.]  If $x=b$, then $\Phi_{\tilde{\xi}_{\inf}}(x;\alpha)=\frac{b-a}{c-a}.$ 
		\item[ IV.] If $b<x< c$, then  
		\[\Phi_{\tilde{\xi}_{\inf}}(x;\alpha)=\left\{
		\begin{array}{ll}
			1-\frac{(c-x)^2}{(c-a)(c-b)}-[(1-\alpha)\theta_l-\alpha\theta_r]\big[\frac{c-b}{c-a}-\frac{(c-x)^2}{(c-a)(c-b)}\big], & b<x\leq c-\frac{c-b}{\sqrt{2}};\\
			1-\frac{(c-x)^2}{(c-a)(c-b)}-[(1-\alpha)\theta_l-\alpha\theta_r]\frac{(c-x)^2}{(c-a)(c-b)}, &  c-\frac{c-b}{\sqrt{2}}\leq x<c.\\
		\end{array}
		\right.
		\] 
		\item[ V.] If $x\geq c$, then  $\Phi_{\tilde{\xi}_{\inf}}(x;\alpha)=1.$ 
	\end{description}
\end{theorem}
\begin{proof}
	Based on Theorem \ref{thm.3}, we have 
	\[\Phi_{\tilde{\xi}_{\inf}}(x;\alpha)=\left\{
	\begin{array}{ll}
		0, & x\leq a;\\
		(1-\alpha)A+\alpha B, &  a<x<b;\\
		\frac{b-a}{c-a}, &  x=b;\\
		(1-\alpha)C+\alpha D, &  b<x<c;\\
		1, &  x\geq c;\\
	\end{array}
	\right.
	 \text{ for any }\alpha \in (0,1).\] 
	 Therefore, we get $\Phi_{\tilde{\xi}_{\inf}}(x;\alpha)=\Phi_{\tilde{\xi}_{\sup}}(x;1-\alpha)$. Using this result, the cases from (I) to (V) can be easily proved from Theorem \ref{thm.6}.
\end{proof}
\begin{theorem}\label{thm.8}
	Let $\tilde{\xi}\sim \mathcal{TRI}(a,b,c;\theta_l,\theta_r)$ be a triangular two-fold UV. The reduced single-fold UD via expected value criteria is as follows.
	\begin{description}
		\item[ I.] If $x\leq a$, then $\Phi_{\tilde{\xi}_{exp}}(x)=0.$  
		\item[ II.] If $a< x< b$, then
		\[\Phi_{\tilde{\xi}_{exp}}(x)=\left\{
		\begin{array}{ll}
			\frac{(x-a)^2}{(b-a)(c-a)}-\frac{(\theta_l-\theta_r)}{2}\cdot\frac{(x-a)^2}{(b-a)(c-a)}, & a<x\leq a+\frac{b-a}{\sqrt{2}};\\
			\frac{(x-a)^2}{(b-a)(c-a)}-\frac{(\theta_l-\theta_r)}{2}\cdot\big[\frac{b-a}{c-a}-\frac{(x-a)^2}{(b-a)(c-a)}\big], & a+\frac{b-a}{\sqrt{2}}\leq x<b.\\
		\end{array}
		\right.
		\]
		\item[ III.]  If $x=b$, then $\Phi_{\tilde{\xi}_{exp}}(x)=\frac{b-a}{c-a}.$ 
		\item[ IV.] If $b< x< c$, then
		\[\Phi_{\tilde{\xi}_{exp}}(x)=\left\{
		\begin{array}{ll}
			1-\frac{(c-x)^2}{(c-a)(c-b)}-\frac{(\theta_l-\theta_r)}{2}\cdot\big[\frac{c-b}{c-a}-\frac{(c-x)^2}{(c-a)(c-b)}\big], & b<x\leq c-\frac{c-b}{\sqrt{2}};\\
			1-\frac{(c-x)^2}{(c-a)(c-b)}-\frac{(\theta_l-\theta_r)}{2}\cdot\frac{(c-x)^2}{(c-a)(c-b)}, &  c-\frac{c-b}{\sqrt{2}}\leq x<c.\\
		\end{array}
		\right.
		\] 
		\item[ V.] If $x\geq c$, then $\Phi_{\tilde{\xi}_{exp}}(x)=1.$     
	\end{description} 
\end{theorem}
\begin{proof}
	Based on Theorem \ref{thm.3}, we have 
	\[\Phi_{\tilde{\xi}_{exp}}(x)=\left\{
	\begin{array}{ll}
		0, & x\leq a;\\
		\frac{A+B}{2}, &  a<x<b;\\
		\frac{b-a}{c-a}, &  x=b;\\
			\frac{C+D}{2}, &  b<x<c;\\
		1, &  x\geq c.\\
	\end{array}
	\right.
	\] 
	Therefore, we get $\Phi_{\tilde{\xi}_{exp}}(x)=\Phi_{\tilde{\xi}_{\inf}}(x;\frac{1}{2})=\Phi_{\tilde{\xi}_{\sup}}(x;	\frac{1}{2})$. With this result, it is easy to get the cases from (I) to (V) from Theorem \ref{thm.6} and Theorem \ref{thm.7}. 
\end{proof}
\begin{remark}
	For any $\alpha\in(0,1),$ $\Phi_{\tilde{\xi}_{\sup}}(x;\alpha)$ is monotonic increasing function except $\Phi_{\tilde{\xi}_{\sup}}(x;\alpha)=0,\forall x\leq a$ and $\Phi_{\tilde{\xi}_{\sup}}(x;\alpha)=1, \forall x\geq c.$ According to Theorem \ref{thm.1}, $\Phi_{\tilde{\xi}_{\sup}}(x;\alpha)$ is a single-fold UD. Similarly, $\Phi_{\tilde{\xi}_{\inf}}(x;\alpha),\Phi_{\tilde{\xi}_{\exp}}(x)$ are also single-fold UDs. As a result, the reduction strategies we provide are acceptable.
	\end{remark}
\par To understand the efficiency of the reduction methods, we provide an example of a triangular two-fold UV below.
\begin{example}
	If $\tilde{\xi}=\mathcal{TRI}(2,4,5;0.5,0.6)$ is a triangular two-fold UV, then
	\begin{description}
		\item[I.] the reduced single-fold UD via $\alpha$ optimistic value criteria is 
		\[\Phi_{\tilde{\xi}_{\sup}}(x;\alpha)=\left\{
		\begin{array}{ll}
			0, & x\leq 2;\\
			\frac{(x-2)^2}{6}-\big(\frac{11}{10}\alpha-\frac{3}{5}\big)\frac{(x-2)^2}{6}, & 2<x\leq 2+\sqrt{2};\\
			\frac{(x-2)^2}{6}-\big(\frac{11}{10}\alpha-\frac{3}{5}\big)\big[\frac{2}{3}-\frac{(x-2)^2}{6}\big], & 2+\sqrt{2}\leq x<4;\\
				\frac{2}{3}, & x=4;\\
					1-\frac{(5-x)^2}{3}-\big(\frac{11}{10}\alpha-\frac{3}{5}\big)\big[\frac{1}{3}-\frac{(5-x)^2}{3}\big], & 4<x\leq 5-\frac{1}{\sqrt{2}} ;\\
				1-\frac{(5-x)^2}{3}-\big(\frac{11}{10}\alpha-\frac{3}{5}\big)\frac{(5-x)^2}{3}, &  5-\frac{1}{\sqrt{2}}\leq x<5;\\
					1, & x\geq 5;\\
		\end{array}
		\right.
	\text{ for any }\alpha \in (0,1);	\]
	\item[II.] the reduced single-fold UD via $\alpha$ pessimistic value criteria is 
	\[\Phi_{\tilde{\xi}_{\inf}}(x;\alpha)=\left\{
	\begin{array}{ll}
		0, & x\leq 2;\\
		\frac{(x-2)^2}{6}-\big(\frac{1}{2}-\frac{11}{10}\alpha\big)\frac{(x-2)^2}{6}, & 2<x\leq 2+\sqrt{2};\\
		\frac{(x-2)^2}{6}-\big(\frac{1}{2}-\frac{11}{10}\alpha\big)\big[\frac{2}{3}-\frac{(x-2)^2}{6}\big], & 2+\sqrt{2}\leq x<4;\\
		\frac{2}{3}, & x=4;\\
		1-\frac{(5-x)^2}{3}-\big(\frac{1}{2}-\frac{11}{10}\alpha\big)\big[\frac{1}{3}-\frac{(5-x)^2}{3}\big], & 4<x\leq 5-\frac{1}{\sqrt{2}};\\
		1-\frac{(5-x)^2}{3}-\big(\frac{1}{2}-\frac{11}{10}\alpha\big)\frac{(5-x)^2}{3}, &  5-\frac{1}{\sqrt{2}}\leq x<5;\\
		1, & x\geq 5;\\
	\end{array}
	\right.
\text{ for any }\alpha \in (0,1);	\]
	\item[III.] the reduced single-fold UD via expected value criteria is 
	\[\Phi_{\tilde{\xi}_{exp}}(x)=\left\{
	\begin{array}{ll}
		0, & x\leq 2;\\
		\frac{(x-2)^2}{6}+\frac{1}{20}\frac{(x-2)^2}{6}, & 2<x\leq 2+\sqrt{2};\\
		\frac{(x-2)^2}{6}+\frac{1}{20}\big[\frac{2}{3}-\frac{(x-2)^2}{6}\big], & 2+\sqrt{2}\leq x<4;\\
		\frac{2}{3}, & x=4;\\
		1-\frac{(5-x)^2}{3}+\frac{1}{20}\big[\frac{1}{3}-\frac{(5-x)^2}{3}\big], & 4<x\leq 5-\frac{1}{\sqrt{2}};\\
		1-\frac{(5-x)^2}{3}+\frac{1}{20}\frac{(5-x)^2}{3}, &  5-\frac{1}{\sqrt{2}}\leq x<5;\\
		1, & x\geq 5.\\
	\end{array}
	\right.
	\]
	\end{description}
	\end{example}
Here, Figures \ref{fig:1},\ref{fig:2}, and \ref{fig:3} show the reduced single-fold UD of the triangular two-fold UV $\tilde{\xi}=\mathcal{TRI}(2,4,5;0.5,0.6)$ using optimistic, pessimistic, and expected value criteria. Figures \ref{fig:1} and \ref{fig:2} are given as solid figures because they are essentially bi-variate functions in terms of $x$ and $\alpha$. For different values of $\alpha\in (0,1)$, the reduced UDs via optimistic and pessimistic value criteria are shown in Figures \ref{fig:4} and \ref{fig:5} in 2D pictures for clarity. It is observed that for each fixed $\alpha$, the related curve satisfies the conditions in Theorem \ref{thm.1}, corresponding to a UD.
\begin{figure}[htb] 
	\centering 
	\includegraphics[scale=0.225]{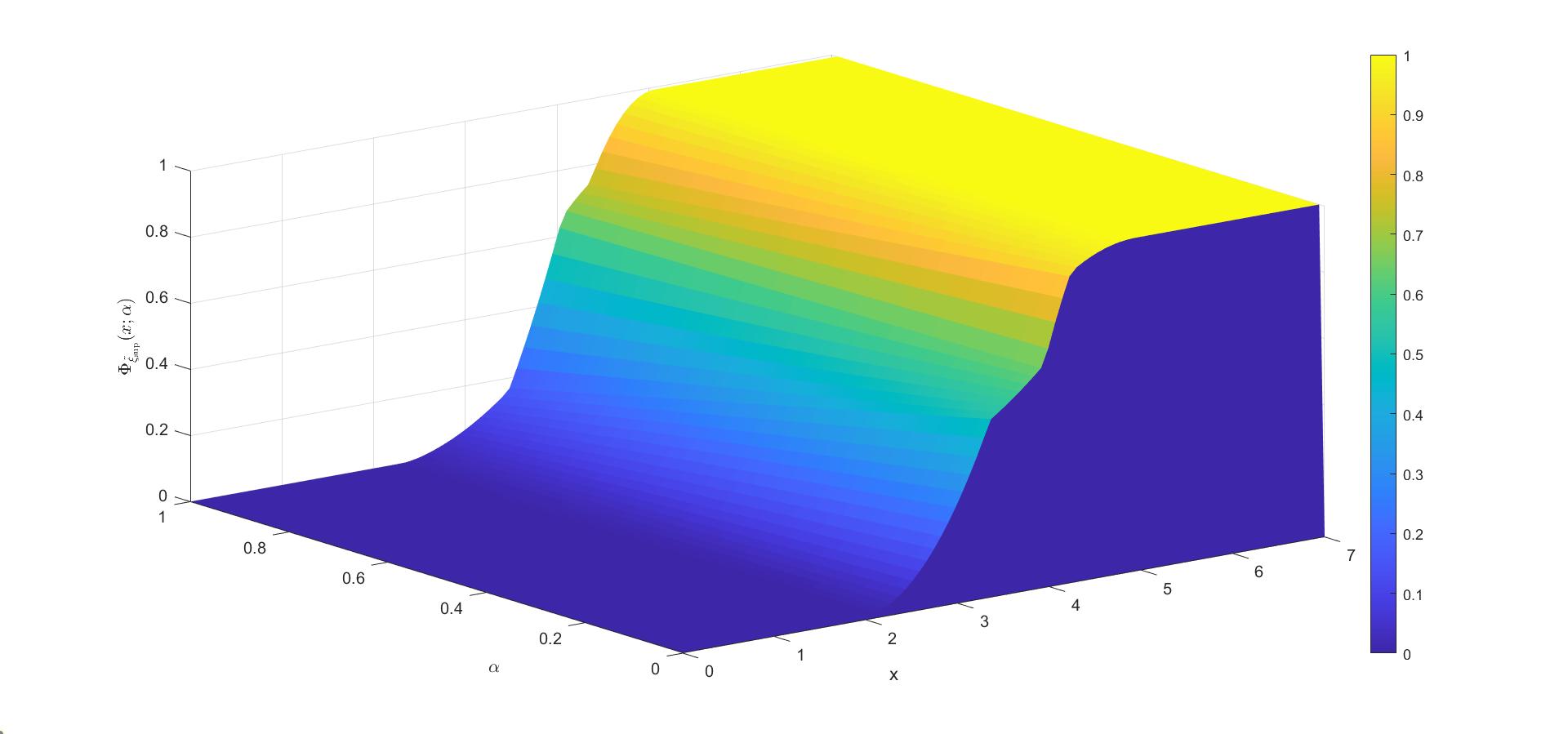} 
	\caption{Reduced single-fold UD of a triangular two-fold UD via $\alpha$ optimistic value criteria.}
	\label{fig:1}
\end{figure}
\begin{figure}[htb] 
	\centering 
	\includegraphics[scale=0.225]{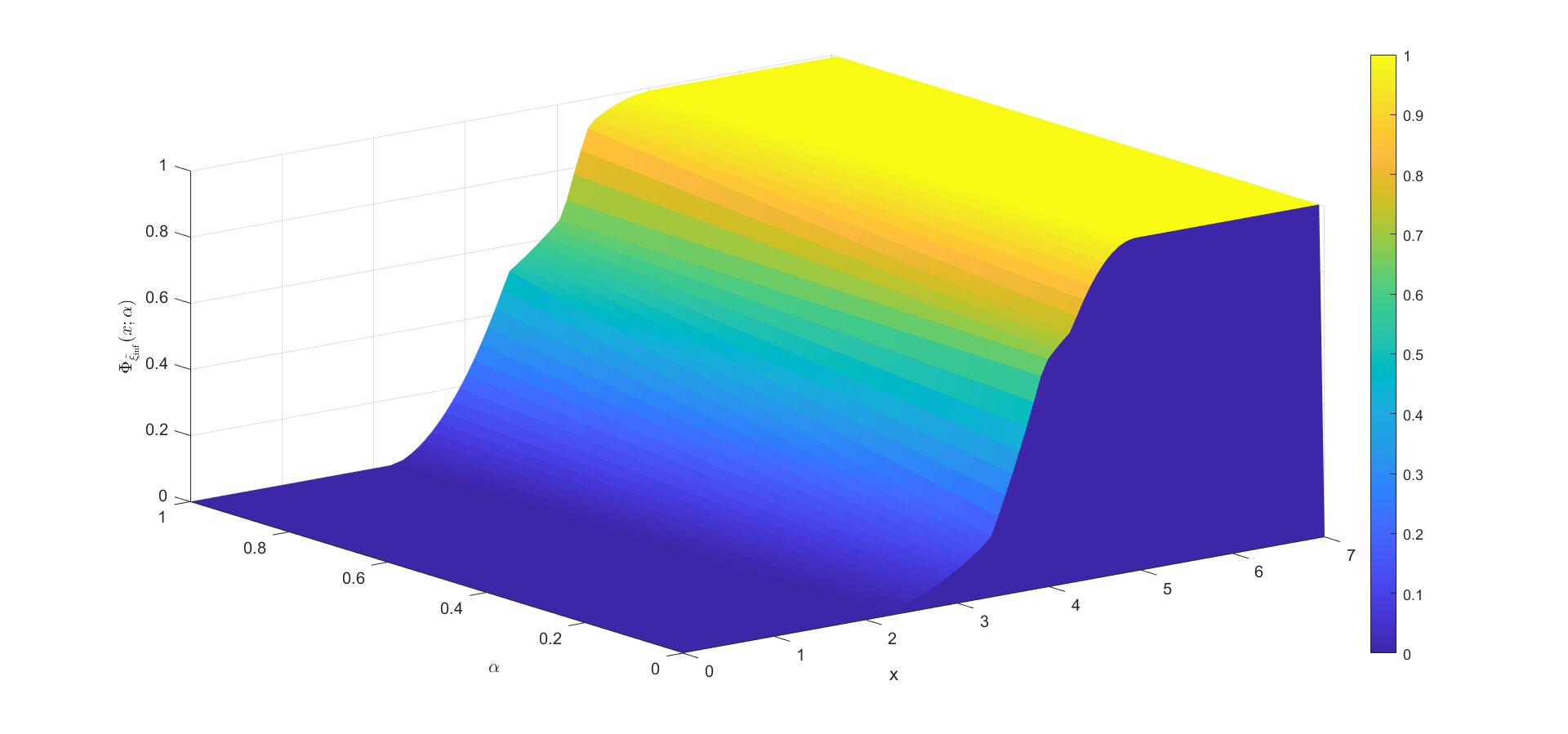} 
	\caption{Reduced single-fold UD of a triangular two-fold UD via $\alpha$ pessimistic value criteria.}
	\label{fig:2}
\end{figure}
\begin{figure}[htb] 
	\centering 
	\includegraphics[scale=0.225]{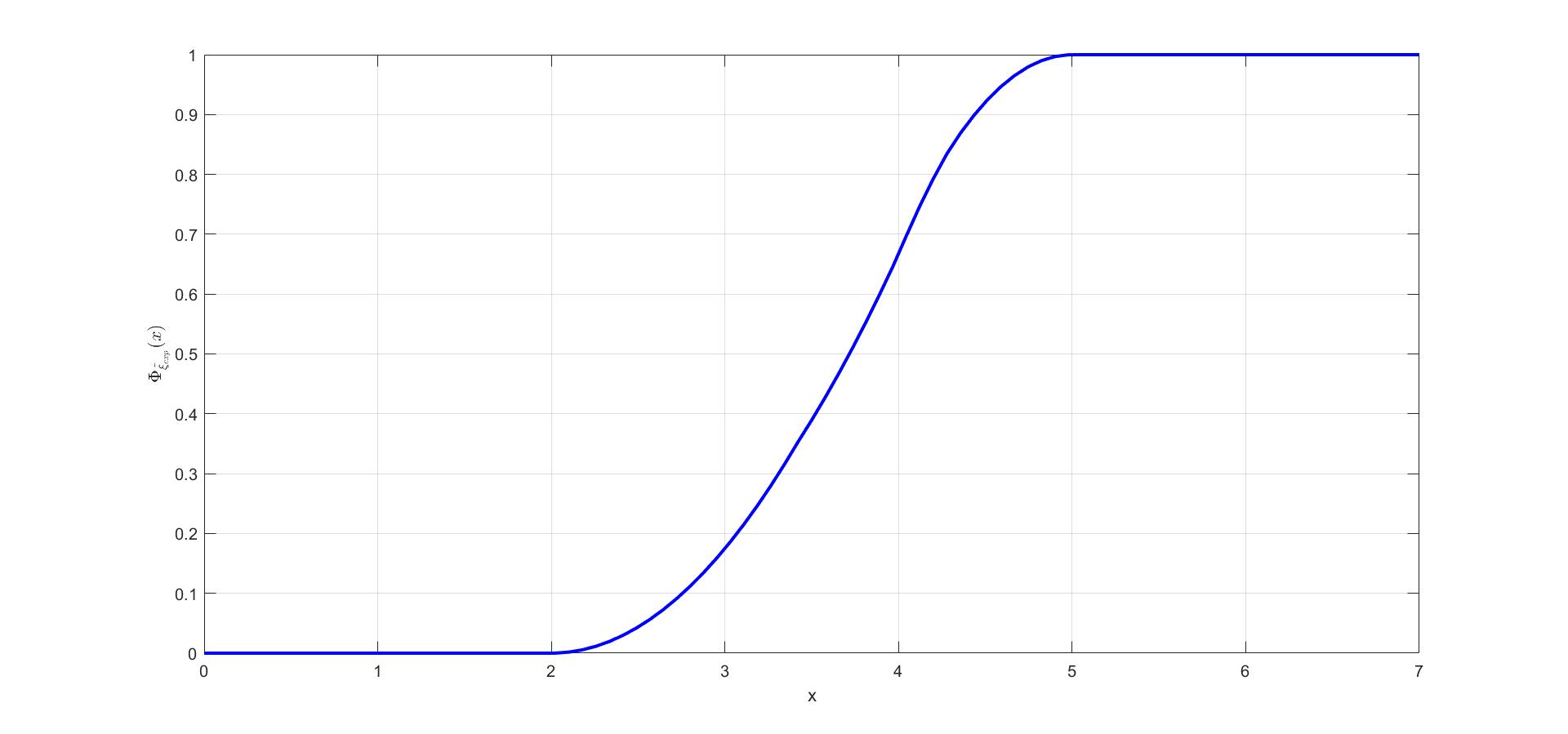} 
	\caption{Reduced single-fold UD of a triangular two-fold UD via expected value criteria.}
	\label{fig:3}
\end{figure}
\begin{figure}[htb]
	\centering
		\includegraphics[scale=0.225]{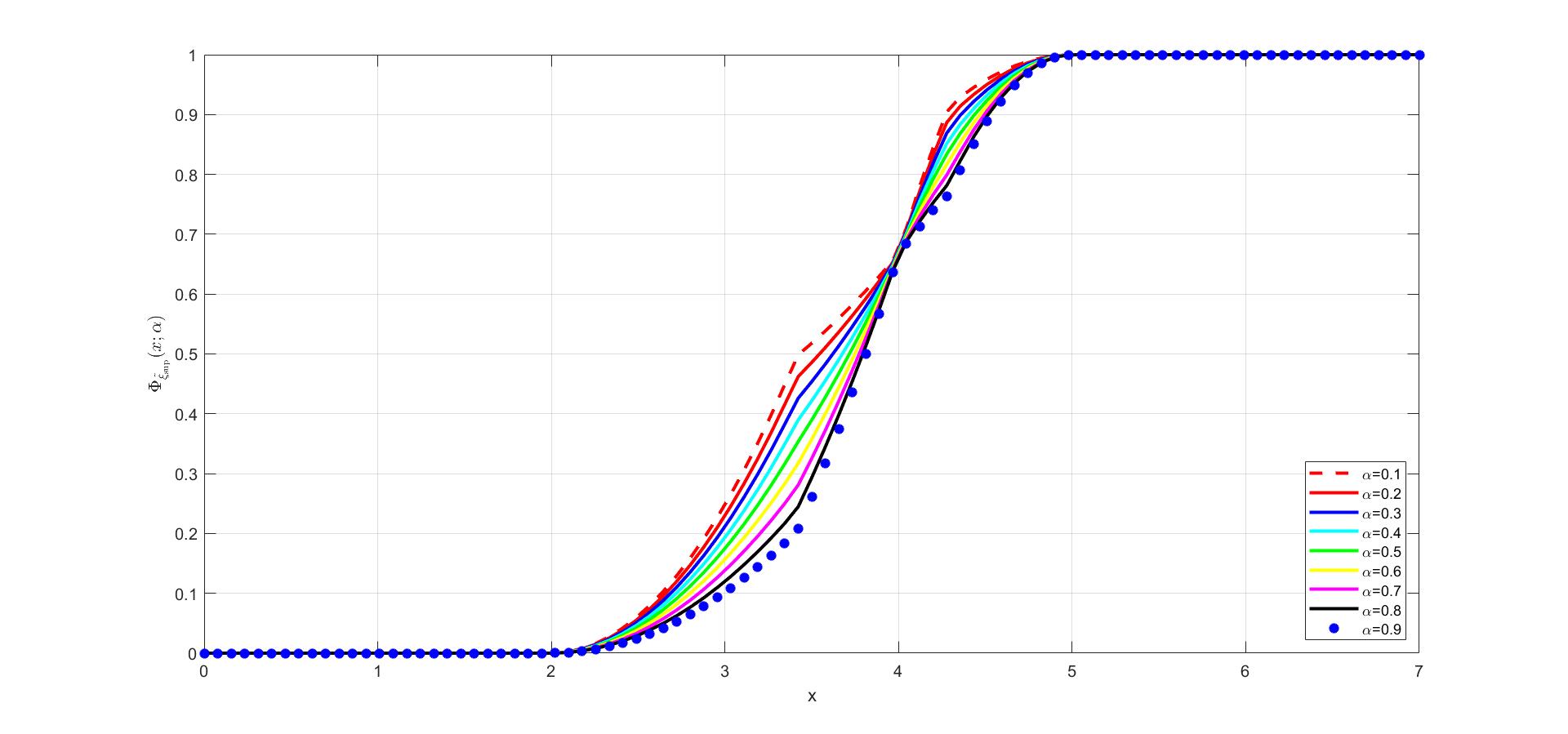}
		\caption{Reduced single-fold UDs of a triangular two-fold UD via optimistic value criteria for different values of $\alpha$.}
		\label{fig:4}
	\end{figure}
	\begin{figure}[htb]
		\centering
		\includegraphics[scale=0.225]{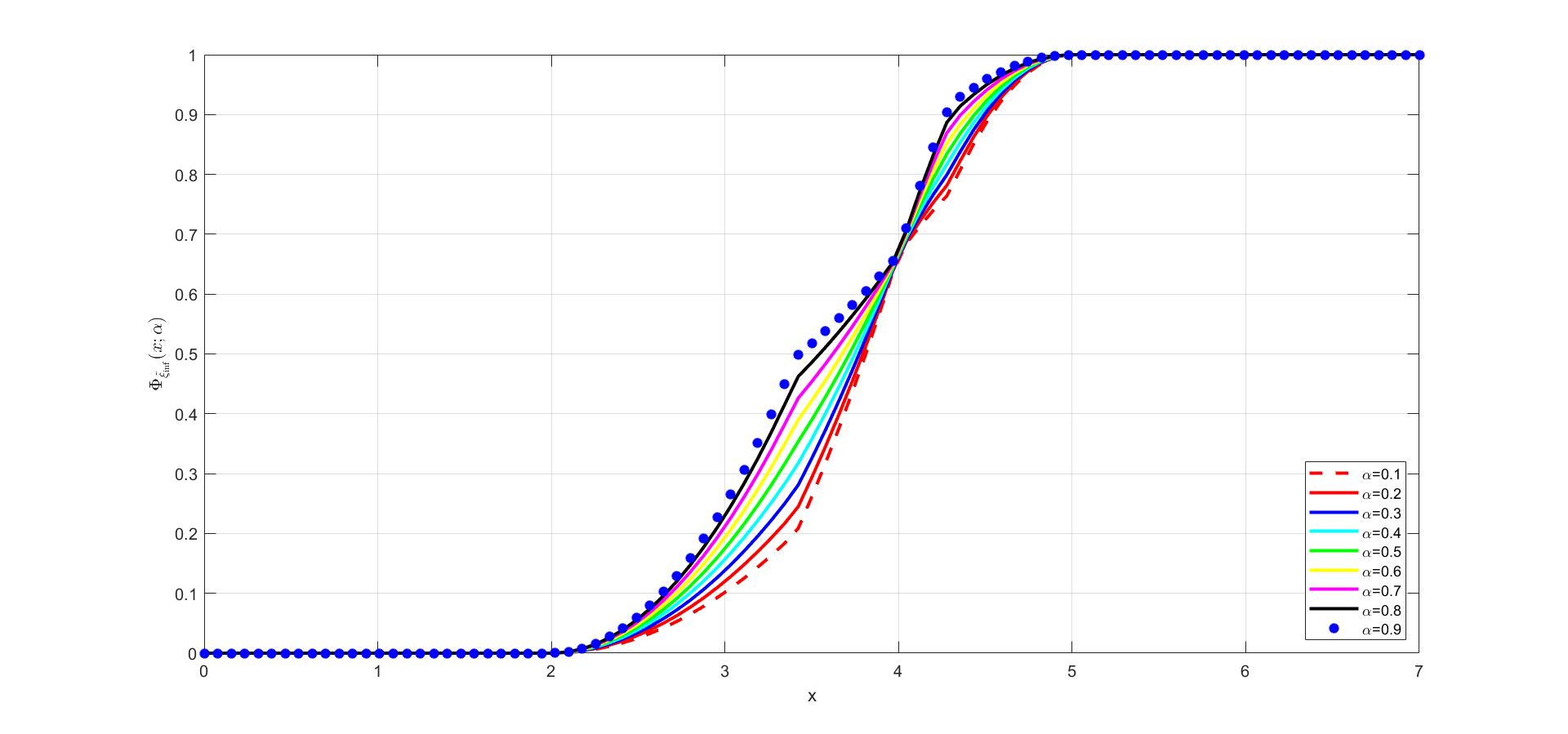}
		\caption{Reduced single-fold UDs of a triangular two-fold UD via pessimistic value criteria for different values of $\alpha$.}
		\label{fig:5}
	\end{figure}
\par Similar to earlier, using optimistic, pessimistic, and expected value criteria, we develop three different reduction methods to convert trapezoidal two-fold UV into single-fold UV. The following theorems are developed in this regard.
\begin{theorem}\label{thm.9}
	Let $\tilde{\xi}\sim \mathcal{TRA}(a,b,c,d;\theta_l,\theta_r)$ be a trapezoidal two-fold UV. The reduced single-fold UD via $\alpha$ optimistic value criteria for any $\alpha \in(0,1)$ is as follows.
	\begin{description}
		\item[ I.] If $x\leq a$, then $\Phi_{\tilde{\xi}_{\sup}}(x;\alpha)=0.$  
		\item[ II.] If $a< x< b$, then 
		\[\Phi_{\tilde{\xi}_{\sup}}(x;\alpha)=\left\{
		\begin{array}{ll}
			\frac{(x-a)^2}{(d+c-a-b)(b-a)}-[\alpha\theta_l-(1-\alpha)\theta_r]\frac{(x-a)^2}{(d+c-a-b)(b-a)}, & a<x\leq a+\frac{b-a}{\sqrt{2}};\\
			\frac{(x-a)^2}{(d+c-a-b)(b-a)}-[\alpha\theta_l-(1-\alpha)\theta_r]\big[\frac{b-a}{d+c-a-b}-\frac{(x-a)^2}{(d+c-a-b)(b-a)}\big], & a+\frac{b-a}{\sqrt{2}}\leq x<b.\\
		\end{array}
		\right.
		\]
		\item[ III.]  If $x=b$, then $\Phi_{\tilde{\xi}_{\sup}}(x;\alpha)=\frac{b-a}{d+c-a-b}.$ 
		\item[ IV.] If $b< x< c$, then
		\[\Phi_{\tilde{\xi}_{\sup}}(x;\alpha)=\left\{
		\begin{array}{ll}
			\frac{2x-a-b}{d+c-a-b}-[\alpha\theta_l-(1-\alpha)\theta_r]\big[\frac{2x-a-b}{d+c-a-b}-\frac{b-a}{d+c-a-b}\big], & b<x\leq \frac{b+c}{2};\\
		\frac{2x-a-b}{d+c-a-b}-[\alpha\theta_l-(1-\alpha)\theta_r]\big[\frac{2c-a-b}{d+c-a-b}-\frac{2x-a-b}{d+c-a-b}\big], &  \frac{b+c}{2}\leq x<c.\\
		\end{array}
		\right.
		\] 
		\item[ V.]  If $x=c$, then $\Phi_{\tilde{\xi}_{\sup}}(x;\alpha)=\frac{2c-a-b}{d+c-a-b}.$ 
		\item[ VI.] If $c< x< d$, then
		\[\Phi_{\tilde{\xi}_{\sup}}(x;\alpha)=\left\{
		\begin{array}{ll}
			1-\frac{(d-x)^2}{(d+c-a-b)(d-c)}-[\alpha\theta_l-(1-\alpha)\theta_r]\big[\frac{d-c}{d+c-a-b}-\frac{(d-x)^2}{(d+c-a-b)(d-c)}\big], & c<x\leq d-\frac{d-c}{\sqrt{2}};\\
			1-\frac{(d-x)^2}{(d+c-a-b)(d-c)}-[\alpha\theta_l-(1-\alpha)\theta_r]\frac{(d-x)^2}{(d+c-a-b)(d-c)}, & d-\frac{d-c}{\sqrt{2}}\leq x<d.\\
		\end{array}
		\right.
		\]
		\item[ VII.] If $x\geq d$, then $\Phi_{\tilde{\xi}_{\sup}}(x;\alpha)=1.$  
	\end{description}
\end{theorem}
\begin{proof}
	Let $\alpha\in (0,1)$.
	 \par If $x\leq a,$ then we have $\tilde{\Phi}_{\tilde{\xi}}(x,y)=0.$ Therefore, $\Phi_{\tilde{\xi}_{\sup}}(x;\alpha)=0, \forall x\leq a.$
	\par If $a< x< b$, then we have $\tilde{\Phi}_{\tilde{\xi}}(x,y)=\mathcal{L}(A,B),$ where $A,B$ are given in Definition \ref{def.12}. So, based on Theorem \ref{thm.3}, we have $\Phi_{\tilde{\xi}_{\sup}}(x;\alpha)=\alpha A+(1-\alpha)B$
	\begin{equation*} 
		\begin{aligned}
			&=\alpha \bigg[\frac{(x-a)^2}{(d+c-a-b)(b-a)}-\theta_l\min\bigg\{\frac{(x-a)^2}{(d+c-a-b)(b-a)},\frac{b-a}{d+c-a-b}-\frac{(x-a)^2}{(d+c-a-b)(b-a)}\bigg\}\bigg]+\\
			&(1-\alpha)\bigg[\frac{(x-a)^2}{(d+c-a-b)(b-a)}+\theta_r\min\bigg\{\frac{(x-a)^2}{(d+c-a-b)(b-a)},\frac{b-a}{d+c-a-b}-\frac{(x-a)^2}{(d+c-a-b)(b-a)}\bigg\}\bigg]\\
			& =\left\{
			\begin{array}{ll}
				\frac{(x-a)^2}{(d+c-a-b)(b-a)}-[\alpha\theta_l-(1-\alpha)\theta_r]\frac{(x-a)^2}{(d+c-a-b)(b-a)}, & a<x\leq a+\frac{b-a}{\sqrt{2}};\\
				\frac{(x-a)^2}{(d+c-a-b)(b-a)}-[\alpha\theta_l-(1-\alpha)\theta_r]\big[\frac{b-a}{d+c-a-b}-\frac{(x-a)^2}{(d+c-a-b)(b-a)}\big], & a+\frac{b-a}{\sqrt{2}}\leq x<b.\\
			\end{array}
			\right.
		\end{aligned}
	\end{equation*}
	\par If $x=b,$ then we have $\tilde{\Phi}_{\tilde{\xi}}(x,y)=\frac{b-a}{d+c-a-b}.$ Therefore, $\Phi_{\tilde{\xi}_{\sup}}(x;\alpha)=\frac{b-a}{d+c-a-b}$ for $x=b.$
	\par If $b< x< c$, then we have $\tilde{\Phi}_{\tilde{\xi}}(x,y)=\mathcal{L}(C,D),$ where $C,D$ are given in Definition \ref{def.12}. So, based on Theorem \ref{thm.3}, we have $	\Phi_{\tilde{\xi}_{\sup}}(x;\alpha)=\alpha C+(1-\alpha)D$
	\begin{equation*} 
		\begin{aligned}
			&=\alpha \bigg[\frac{2x-a-b}{d+c-a-b}-\theta_l\min\bigg\{\frac{2x-a-b}{d+c-a-b}-\frac{b-a}{d+c-a-b},\frac{2c-a-b}{d+c-a-b}-\frac{2x-a-b}{d+c-a-b}\bigg\}\bigg]+\\
			&(1-\alpha)\bigg[\frac{2x-a-b}{d+c-a-b}+\theta_r\min\bigg\{\frac{2x-a-b}{d+c-a-b}-\frac{b-a}{d+c-a-b},\frac{2c-a-b}{d+c-a-b}-\frac{2x-a-b}{d+c-a-b}\bigg\}\bigg]\\
			& =\left\{
			\begin{array}{ll}
				\frac{2x-a-b}{d+c-a-b}-[\alpha\theta_l-(1-\alpha)\theta_r]\big[\frac{2x-a-b}{d+c-a-b}-\frac{b-a}{d+c-a-b}\big], & b<x\leq \frac{b+c}{2};\\
				\frac{2x-a-b}{d+c-a-b}-[\alpha\theta_l-(1-\alpha)\theta_r]\big[\frac{2c-a-b}{d+c-a-b}-\frac{2x-a-b}{d+c-a-b}\big], &  \frac{b+c}{2}\leq x<c.\\
			\end{array}
			\right.
		\end{aligned}
	\end{equation*}
\par If $x=c,$ then we have $\tilde{\Phi}_{\tilde{\xi}}(x,y)=\frac{2c-a-b}{d+c-a-b}.$ Therefore, $\Phi_{\tilde{\xi}_{\sup}}(x;\alpha)=\frac{2c-a-b}{d+c-a-b}$ for $x=c.$
\par If $c< x< d$, then we have $\tilde{\Phi}_{\tilde{\xi}}(x,y)=\mathcal{L}(E,F),$ where $E,F$ are given in Definition \ref{def.12}. So, based on Theorem \ref{thm.3}, we have $\Phi_{\tilde{\xi}_{\sup}}(x;\alpha)=\alpha E+(1-\alpha)F$
\begin{equation*} 
	\begin{aligned}
		&=\alpha \bigg[1-\frac{(d-x)^2}{(d+c-a-b)(d-c)}-\theta_l\min\bigg\{\frac{d-c}{d+c-a-b}-\frac{(d-x)^2}{(d+c-a-b)(d-c)},\frac{(d-x)^2}{(d+c-a-b)(d-c)}\bigg\}\bigg]+\\
		&(1-\alpha)\bigg[1-\frac{(d-x)^2}{(d+c-a-b)(d-c)}+\theta_r\min\bigg\{\frac{d-c}{d+c-a-b}-\frac{(d-x)^2}{(d+c-a-b)(d-c)},\frac{(d-x)^2}{(d+c-a-b)(d-c)}\bigg\}\bigg]\\
		& =\left\{
		\begin{array}{ll}
			1-\frac{(d-x)^2}{(d+c-a-b)(d-c)}-[\alpha\theta_l-(1-\alpha)\theta_r]\big[\frac{d-c}{d+c-a-b}-\frac{(d-x)^2}{(d+c-a-b)(d-c)}\big], & c<x\leq d-\frac{d-c}{\sqrt{2}};\\
			1-\frac{(d-x)^2}{(d+c-a-b)(d-c)}-[\alpha\theta_l-(1-\alpha)\theta_r]\frac{(d-x)^2}{(d+c-a-b)(d-c)}, & d-\frac{d-c}{\sqrt{2}}\leq x<d.\\
		\end{array}
		\right.
	\end{aligned}
\end{equation*}
	\par 	If $x\geq d,$ then we have $\tilde{\Phi}_{\tilde{\xi}}(x,y)=1.$ Therefore, $\Phi_{\tilde{\xi}_{\sup}}(x;\alpha)=1, \forall x\geq d.$
\end{proof}
\begin{theorem}\label{thm.10}
	Let $\tilde{\xi}\sim \mathcal{TRA}(a,b,c,d;\theta_l,\theta_r)$ be a trapezoidal two-fold UV. The reduced single-fold UD via $\alpha$ pessimistic value criteria for any $\alpha \in(0,1)$ is as follows.
	\begin{description}
		\item[ I.] If $x\leq a$, then $\Phi_{\tilde{\xi}_{\inf}}(x;\alpha)=0.$  
		\item[ II.] If $a< x< b$, then 
		\[\Phi_{\tilde{\xi}_{\inf}}(x;\alpha)=\left\{
		\begin{array}{ll}
			\frac{(x-a)^2}{(d+c-a-b)(b-a)}-[(1-\alpha)\theta_l-\alpha\theta_r]\frac{(x-a)^2}{(d+c-a-b)(b-a)}, & a<x\leq a+\frac{b-a}{\sqrt{2}};\\
			\frac{(x-a)^2}{(d+c-a-b)(b-a)}-[(1-\alpha)\theta_l-\alpha\theta_r]\big[\frac{b-a}{d+c-a-b}-\frac{(x-a)^2}{(d+c-a-b)(b-a)}\big], & a+\frac{b-a}{\sqrt{2}}\leq x<b.\\
		\end{array}
		\right.
		\]
		\item[ III.]  If $x=b$, then $\Phi_{\tilde{\xi}_{\inf}}(x;\alpha)=\frac{b-a}{d+c-a-b}.$ 
		\item[ IV.] If $b< x< c$, then
		\[\Phi_{\tilde{\xi}_{\inf}}(x;\alpha)=\left\{
		\begin{array}{ll}
			\frac{2x-a-b}{d+c-a-b}-[(1-\alpha)\theta_l-\alpha\theta_r]\big[\frac{2x-a-b}{d+c-a-b}-\frac{b-a}{d+c-a-b}\big], & b<x\leq \frac{b+c}{2};\\
			\frac{2x-a-b}{d+c-a-b}-[(1-\alpha)\theta_l-\alpha\theta_r]\big[\frac{2c-a-b}{d+c-a-b}-\frac{2x-a-b}{d+c-a-b}\big], &  \frac{b+c}{2}\leq x<c.\\
		\end{array}
		\right.
		\] 
		\item[ V.]  If $x=c$, then $\Phi_{\tilde{\xi}_{\inf}}(x;\alpha)=\frac{2c-a-b}{d+c-a-b}.$ 
		\item[ VI.] If $c< x< d$, then
		\[\Phi_{\tilde{\xi}_{\inf}}(x;\alpha)=\left\{
		\begin{array}{ll}
			1-\frac{(d-x)^2}{(d+c-a-b)(d-c)}-[(1-\alpha)\theta_l-\alpha\theta_r]\big[\frac{d-c}{d+c-a-b}-\frac{(d-x)^2}{(d+c-a-b)(d-c)}\big], & c<x\leq d-\frac{d-c}{\sqrt{2}};\\
			1-\frac{(d-x)^2}{(d+c-a-b)(d-c)}-[(1-\alpha)\theta_l-\alpha\theta_r]\frac{(d-x)^2}{(d+c-a-b)(d-c)}, & d-\frac{d-c}{\sqrt{2}}\leq x<d.\\
		\end{array}
		\right.
		\]
		\item[ VII.] If $x\geq d$, then $\Phi_{\tilde{\xi}_{\inf}}(x;\alpha)=1.$  
	\end{description}
\end{theorem}
\begin{proof}
	Based on Theorem \ref{thm.3}, we have 
	\[\Phi_{\tilde{\xi}_{\inf}}(x;\alpha)=\left\{
	\begin{array}{ll}
		0, & x\leq a;\\
		(1-\alpha)A+\alpha B, &  a<x<b;\\
		\frac{b-a}{d+c-a-b}, &  x=b;\\
		(1-\alpha)C+\alpha D, &  b<x<c;\\
		\frac{2c-a-b}{d+c-a-b}, &  x=c;\\
		(1-\alpha)E+\alpha F, &  c<x<d;\\
		1, &  x\geq d;\\
	\end{array}
	\right.
	\text{ for any }\alpha \in (0,1).\] 
	Therefore, we get $\Phi_{\tilde{\xi}_{\inf}}(x;\alpha)=\Phi_{\tilde{\xi}_{\sup}}(x;1-\alpha)$. Using this result, the cases from (I) to (VII) can be easily proved from Theorem \ref{thm.9}.
\end{proof}
\begin{theorem}\label{thm.11}
	Let $\tilde{\xi}\sim \mathcal{TRA}(a,b,c,d;\theta_l,\theta_r)$ be a trapezoidal two-fold UV. The reduced single-fold UD via expected value criteria is as follows.
	\begin{description}
		\item[ I.] If $x\leq a$, then $\Phi_{\tilde{\xi}_{exp}}(x)=0.$  
		\item[ II.] If $a< x< b$, then 
		\[\Phi_{\tilde{\xi}_{exp}}(x)=\left\{
		\begin{array}{ll}
			\frac{(x-a)^2}{(d+c-a-b)(b-a)}-\frac{(\theta_l-\theta_r)}{2}\cdot\frac{(x-a)^2}{(d+c-a-b)(b-a)}, & a<x\leq a+\frac{b-a}{\sqrt{2}};\\
			\frac{(x-a)^2}{(d+c-a-b)(b-a)}-\frac{(\theta_l-\theta_r)}{2}\cdot\big[\frac{b-a}{d+c-a-b}-\frac{(x-a)^2}{(d+c-a-b)(b-a)}\big], & a+\frac{b-a}{\sqrt{2}}\leq x<b.\\
		\end{array}
		\right.
		\]
		\item[ III.]  If $x=b$, then $\Phi_{\tilde{\xi}_{exp}}(x)=\frac{b-a}{d+c-a-b}.$ 
		\item[ IV.] If $b< x< c$, then
		\[\Phi_{\tilde{\xi}_{exp}}(x)=\left\{
		\begin{array}{ll}
			\frac{2x-a-b}{d+c-a-b}-\frac{(\theta_l-\theta_r)}{2}\cdot\big[\frac{2x-a-b}{d+c-a-b}-\frac{b-a}{d+c-a-b}\big], & b<x\leq \frac{b+c}{2};\\
			\frac{2x-a-b}{d+c-a-b}-\frac{(\theta_l-\theta_r)}{2}\cdot\big[\frac{2c-a-b}{d+c-a-b}-\frac{2x-a-b}{d+c-a-b}\big], &  \frac{b+c}{2}\leq x<c.\\
		\end{array}
		\right.
		\] 
		\item[ V.]  If $x=c$, then $\Phi_{\tilde{\xi}_{exp}}(x)=\frac{2c-a-b}{d+c-a-b}.$ 
		\item[ VI.] If $c< x< d$, then
		\[\Phi_{\tilde{\xi}_{exp}}(x)=\left\{
		\begin{array}{ll}
			1-\frac{(d-x)^2}{(d+c-a-b)(d-c)}-\frac{(\theta_l-\theta_r)}{2}\cdot\big[\frac{d-c}{d+c-a-b}-\frac{(d-x)^2}{(d+c-a-b)(d-c)}\big], & c<x\leq d-\frac{d-c}{\sqrt{2}};\\
			1-\frac{(d-x)^2}{(d+c-a-b)(d-c)}-\frac{(\theta_l-\theta_r)}{2}\cdot\frac{(d-x)^2}{(d+c-a-b)(d-c)}, & d-\frac{d-c}{\sqrt{2}}\leq x<d.\\
		\end{array}
		\right.
		\]
		\item[ VII.] If $x\geq d$, then $\Phi_{\tilde{\xi}_{exp}}(x)=1.$  
	\end{description}
\end{theorem}
\begin{proof}
	Based on Theorem \ref{thm.3}, we have 
	\[\Phi_{\tilde{\xi}_{exp}}(x)=\left\{
	\begin{array}{ll}
		0, & x\leq a;\\
		\frac{A+B}{2}, &  a<x<b;\\
		\frac{b-a}{d+c-a-b}, &  x=b;\\
		\frac{C+D}{2}, &  b<x<c;\\
		\frac{2c-a-b}{d+c-a-b}, &  x=c;\\
		\frac{E+F}{2}, &  c<x<d;\\
		1, &  x\geq d.\\
	\end{array}
	\right.
	\] 
	Therefore, we get $\Phi_{\tilde{\xi}_{exp}}(x)=\Phi_{\tilde{\xi}_{\inf}}(x;\frac{1}{2})=\Phi_{\tilde{\xi}_{\sup}}(x;	\frac{1}{2})$. With this result, it is easy to get the cases from (I) to (VII) from Theorem \ref{thm.9} and Theorem \ref{thm.10}. 
\end{proof}
\begin{remark}
	Similar to earlier remark, for any $\alpha\in(0,1),$ $\Phi_{\tilde{\xi}_{\sup}}(x;\alpha)$ is monotonic increasing function except $\Phi_{\tilde{\xi}_{\sup}}(x;\alpha)=0,\forall x\leq a$ and $\Phi_{\tilde{\xi}_{\sup}}(x;\alpha)=1, \forall x\geq d.$ According to Theorem \ref{thm.1}, $\Phi_{\tilde{\xi}_{\sup}}(x;\alpha)$ is a single-fold UD. Similarly, $\Phi_{\tilde{\xi}_{\inf}}(x;\alpha),\Phi_{\tilde{\xi}_{\exp}}(x)$ are also single-fold UDs. As a result, our recommended strategies for reduction are acceptable.
\end{remark}
\par Here also, we provide an example of a trapezoidal two-fold UV to understand the efficiency of the reduction methods.
\begin{example}
	If $\tilde{\xi}=\mathcal{TRA}(2,4,6,8;0.5,0.6)$ is a trapezoidal two-fold UV, then
	\begin{description}
		\item[I.] the reduced single-fold UD via $\alpha$ optimistic value criteria is 
		\[\Phi_{\tilde{\xi}_{\sup}}(x;\alpha)=\left\{
		\begin{array}{ll}
			0, & x\leq 2;\\
			\frac{(x-2)^2}{16}-\big(\frac{11}{10}\alpha-\frac{3}{5}\big)\frac{(x-2)^2}{16}, & 2<x\leq 2+\sqrt{2};\\
			\frac{(x-2)^2}{16}-\big(\frac{11}{10}\alpha-\frac{3}{5}\big)\big[\frac{1}{4}-\frac{(x-2)^2}{16}\big], & 2+\sqrt{2}\leq x<4;\\
			\frac{1}{4}, & x=4;\\
			\frac{2x-6}{8}-\big(\frac{11}{10}\alpha-\frac{3}{5}\big)\big[\frac{2x-6}{8}-\frac{1}{4}\big], & 4<x\leq 5;\\
			\frac{2x-6}{8}-\big(\frac{11}{10}\alpha-\frac{3}{5}\big)\big[\frac{3}{4}-\frac{2x-6}{8}\big], &  5\leq x<6;\\
			\frac{3}{4}, & x=6;\\
			1-\frac{(8-x)^2}{16}-\big(\frac{11}{10}\alpha-\frac{3}{5}\big)\big[\frac{1}{4}-\frac{(8-x)^2}{16}\big], & 6<x\leq 8-\sqrt{2};\\
			1-\frac{(8-x)^2}{16}-\big(\frac{11}{10}\alpha-\frac{3}{5}\big)\frac{(8-x)^2}{16}, &  8-\sqrt{2}\leq x<8;\\
			1, & x\geq 8;\\
		\end{array}
		\right.
		\text{ for any }\alpha \in (0,1);	\]
		\item[II.] the reduced single-fold UD via $\alpha$ pessimistic value criteria is 
		\[\Phi_{\tilde{\xi}_{\inf}}(x;\alpha)=\left\{
		\begin{array}{ll}
			0, & x\leq 2;\\
			\frac{(x-2)^2}{16}-\big(\frac{1}{2}-\frac{11}{10}\alpha\big)\frac{(x-2)^2}{16}, & 2<x\leq 2+\sqrt{2};\\
			\frac{(x-2)^2}{16}-\big(\frac{1}{2}-\frac{11}{10}\alpha\big)\big[\frac{1}{4}-\frac{(x-2)^2}{16}\big], & 2+\sqrt{2}\leq x<4;\\
			\frac{1}{4}, & x=4;\\
			\frac{2x-6}{8}-\big(\frac{1}{2}-\frac{11}{10}\alpha\big)\big[\frac{2x-6}{8}-\frac{1}{4}\big], & 4<x\leq 5;\\
			\frac{2x-6}{8}-\big(\frac{1}{2}-\frac{11}{10}\alpha\big)\big[\frac{3}{4}-\frac{2x-6}{8}\big], &  5\leq x<6;\\
			\frac{3}{4}, & x=6;\\
			1-\frac{(8-x)^2}{16}-\big(\frac{1}{2}-\frac{11}{10}\alpha\big)\big[\frac{1}{4}-\frac{(8-x)^2}{16}\big], & 6<x\leq 8-\sqrt{2};\\
			1-\frac{(8-x)^2}{16}-\big(\frac{1}{2}-\frac{11}{10}\alpha\big)\frac{(8-x)^2}{16}, &  8-\sqrt{2}\leq x<8;\\
			1, & x\geq 8;\\
		\end{array}
		\right.
		\text{ for any }\alpha \in (0,1);	\]
		\item[III.] the reduced single-fold UD via expected value criteria is 
		\[\Phi_{\tilde{\xi}_{exp}}(x)=\left\{
		\begin{array}{ll}
			0, & x\leq 2;\\
			\frac{(x-2)^2}{16}+\frac{1}{20}\frac{(x-2)^2}{16}, & 2<x\leq 2+\sqrt{2};\\
			\frac{(x-2)^2}{16}+\frac{1}{20}\big[\frac{1}{4}-\frac{(x-2)^2}{16}\big], & 2+\sqrt{2}\leq x<4;\\
			\frac{1}{4}, & x=4;\\
			\frac{2x-6}{8}+\frac{1}{20}\big[\frac{2x-6}{8}-\frac{1}{4}\big], & 4<x\leq 5;\\
			\frac{2x-6}{8}+\frac{1}{20}\big[\frac{3}{4}-\frac{2x-6}{8}\big], &  5\leq x<6;\\
			\frac{3}{4}, & x=6;\\
			1-\frac{(8-x)^2}{16}+\frac{1}{20}\big[\frac{1}{4}-\frac{(8-x)^2}{16}\big], & 6<x\leq 8-\sqrt{2};\\
			1-\frac{(8-x)^2}{16}+\frac{1}{20}\frac{(8-x)^2}{16}, &  8-\sqrt{2}\leq x<8;\\
			1, & x\geq 8.\\
		\end{array}
		\right.
		\]
	\end{description}
\end{example}
Similar to earlier, Figures \ref{fig:6},\ref{fig:7}, and \ref{fig:8} show the reduced single-fold UD of a trapezoidal two-fold UV $\tilde{\xi}=\mathcal{TRA}(2,4,6,8;0.5,0.6)$ using optimistic, pessimistic, and expected value criteria. Figures \ref{fig:6} and \ref{fig:7} in particular are presented as solid figures because they are essentially bi-variate functions with respect to $x$ and $\alpha$. For different values of $\alpha\in (0,1)$, the reduced UDs via optimistic and pessimistic value criteria are shown in Figures \ref{fig:9} and \ref{fig:10} in 2D pictures for clarity. It is observed that for each fixed $\alpha$, the related curve satisfies the conditions in Theorem \ref{thm.1}, corresponding to a UD.
\begin{figure}[htb] 
	\centering 
	\includegraphics[scale=0.225]{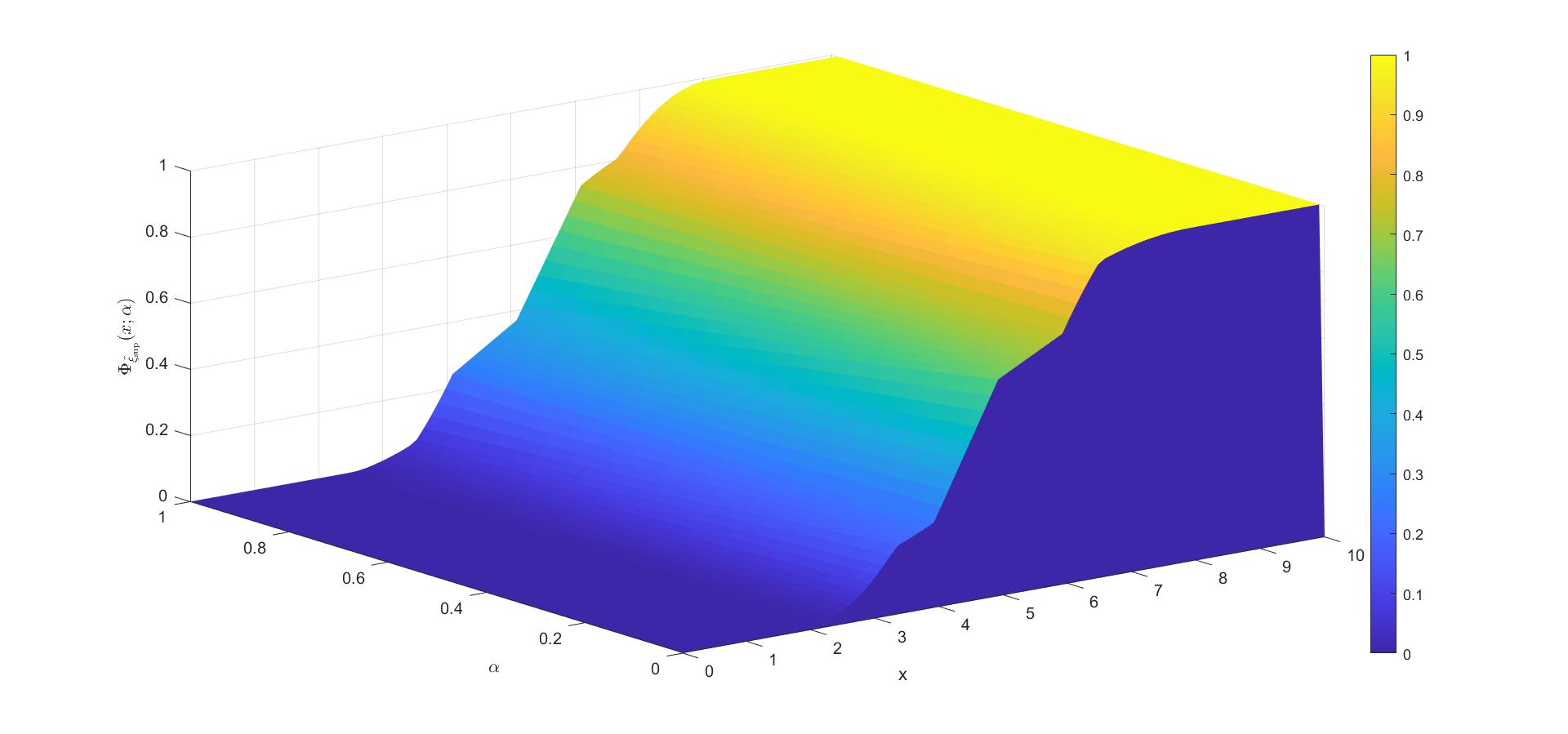} 
	\caption{Reduced single-fold UD of a trapezoidal two-fold UD via $\alpha$ optimistic value criteria.}
	\label{fig:6}
\end{figure}
\begin{figure}[htb] 
	\centering 
	\includegraphics[scale=0.225]{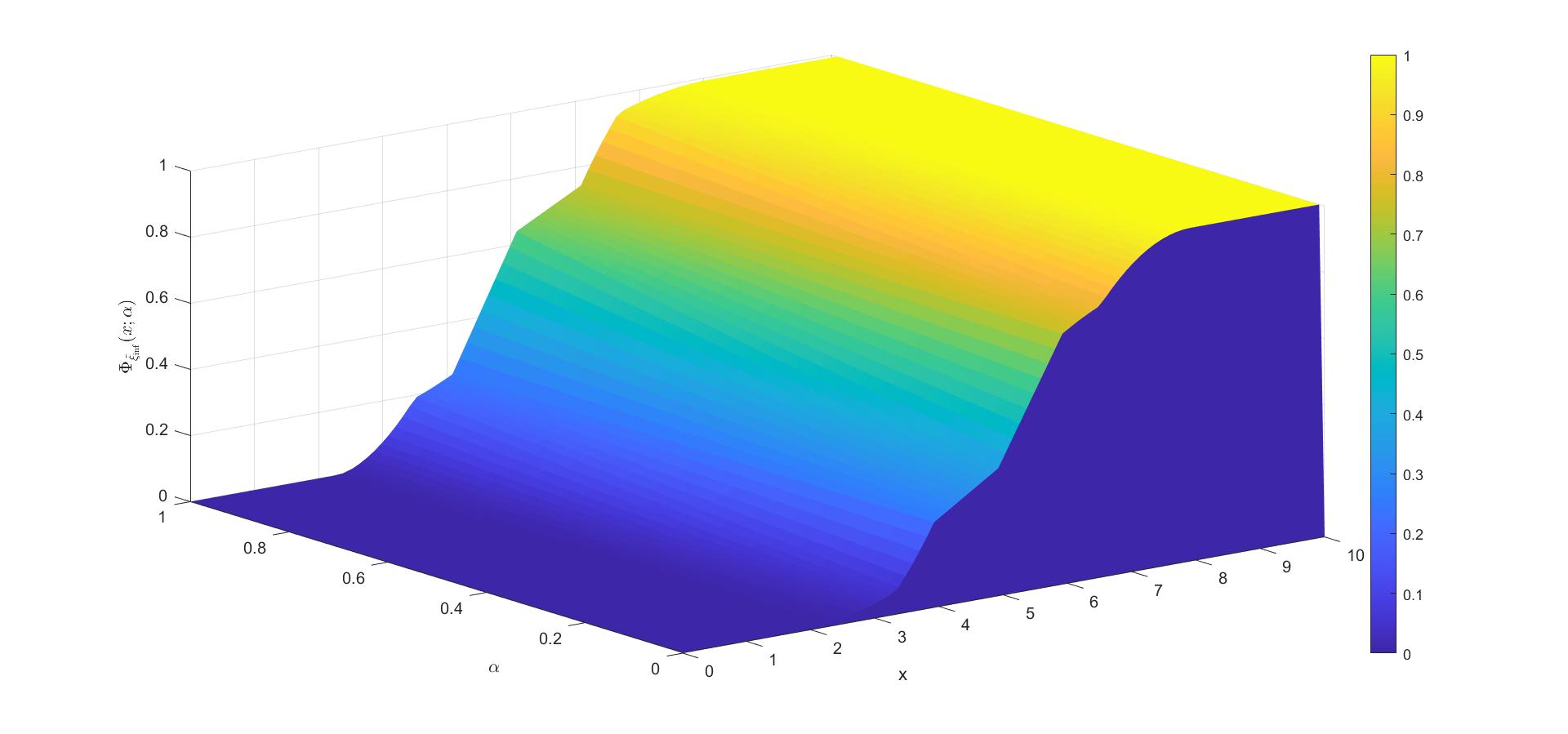} 
	\caption{Reduced single-fold UD of a trapezoidal two-fold UD via $\alpha$ pessimistic value criteria.}
	\label{fig:7}
\end{figure}
\begin{figure}[htb] 
	\centering 
	\includegraphics[scale=0.225]{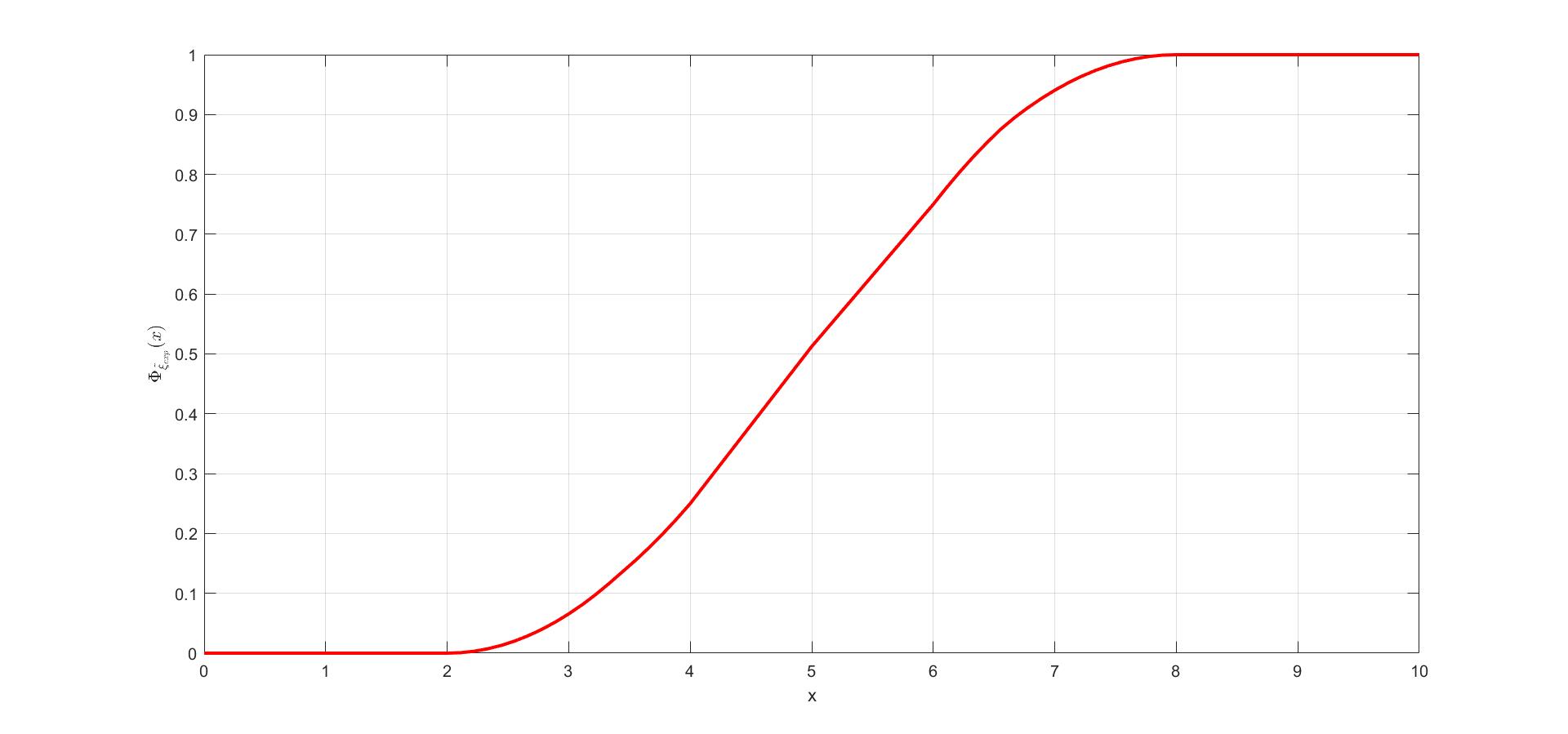} 
	\caption{Reduced single-fold UD of a trapezoidal two-fold UD via expected value criteria.}
	\label{fig:8}
\end{figure}
\begin{figure}[htb]
	\centering
		\includegraphics[scale=0.225]{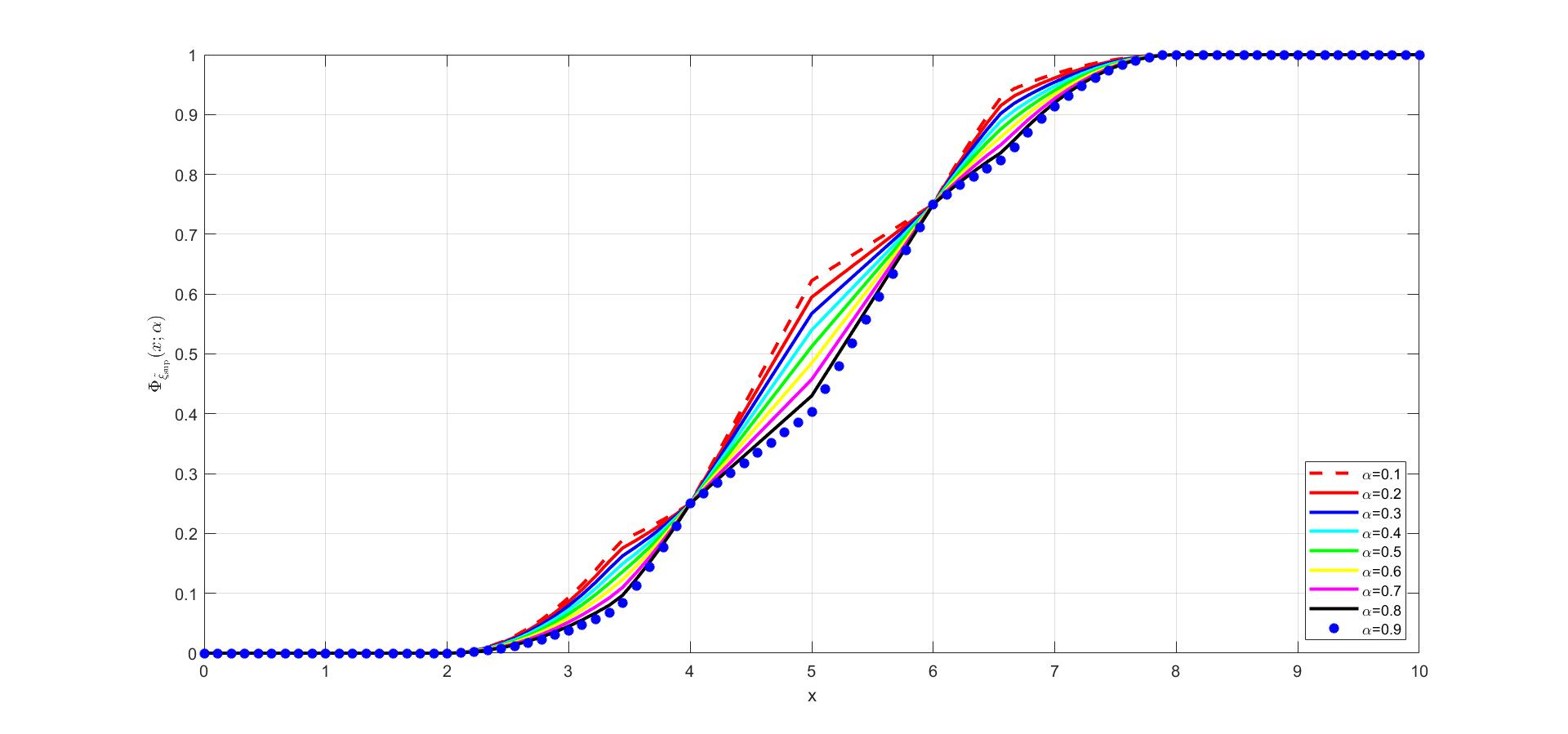}
		\caption{Reduced single-fold UDs of a trapezoidal two-fold UD via optimistic value criteria for different values of $\alpha$.}
		\label{fig:9}
	\end{figure}
	\begin{figure}[htb]
		\centering
		\includegraphics[scale=0.225]{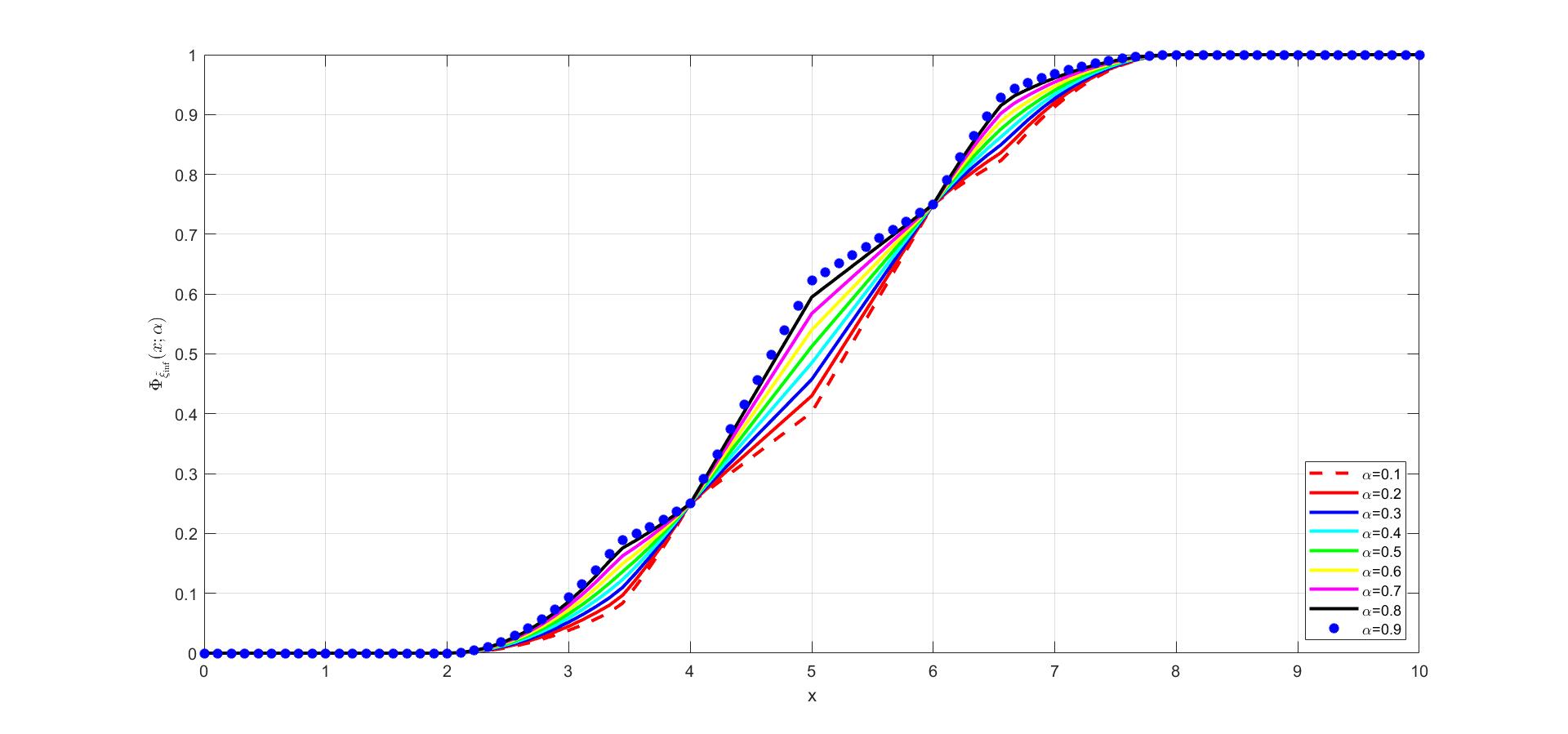}
		\caption{Reduced single-fold UDs of a trapezoidal two-fold UD via pessimistic value criteria for different values of $\alpha$.}
		\label{fig:10}
	\end{figure}
\section{ Geometric Programming Problem with Two-fold Uncertain Coefficients }\label{sec.4} 
The conventional GP model assumes that the coefficients are precise and exact. However, in the real-world GP model, the coefficients may be imprecise and ambiguous. To overcome the difficulty, we propose to solve the GP problem in an uncertain environment with the coefficients being triangular and trapezoidal two-fold UVs. To do so, first, we develop the conventional GP problem into a GP problem with coefficients that are two-fold UVs. In an uncertain environment, the GP problem can be formulated as
	\begin{equation}\label{eq.1}
	\begin{aligned}
		&	\min \quad \tilde{f}_0{(\textbf{x})}=\sum\limits_{i=1}^{N_0}{\tilde{\beta}_{i0}}^2\prod\limits_{j=1}^{n}x_j^{\alpha_{0ij}}\\
		&	\text{s.t.}\\
		& \qquad \tilde{f}_k{(\textbf{x})}=\sum\limits_{i=1}^{N_k}{\tilde{\beta}_{ik}}^2\prod\limits_{j=1}^{n}x_j^{\alpha_{kij}}\leq 1,k=1,2,\ldots K,
	\end{aligned}
\end{equation}
where the coefficients ${\tilde{\beta}_{ik}}^2$ are independent two-fold UVs, $x_j>0$, $\alpha_{kij} \in \mathbb{R},\forall i,j,k$. 
\par The goal of this study is to create the equivalent deterministic form of Problem \ref{eq.1}. To do so, we must first convert two-fold UVs ${\tilde{\beta}_{ik}}^2$ into single-fold UVs using optimistic, pessimistic, and expected value criteria. If ${\tilde{\beta}_{ik}}^2$ is triangular two-fold UV, then it can be reduced into single-fold UV via $\alpha$ optimistic, $\alpha$ pessimistic, and expected value criteria based on Theorem \ref{thm.6}, Theorem \ref{thm.7}, and Theorem \ref{thm.8}, respectively. Similarly, if ${\tilde{\beta}_{ik}}^2$ is trapezoidal two-fold UV, then it can be reduced into single-fold UV via $\alpha$ optimistic, $\alpha$ pessimistic, and expected value criteria based on Theorem \ref{thm.9}, Theorem \ref{thm.10}, and Theorem \ref{thm.11}, respectively. Let ${\tilde{\beta}_{ik}}^1$ be the reduced single-fold UV with the reduced single-fold UD $\Phi_{\tilde{\beta}_{ik}^1}$. After reducing ${\tilde{\beta}_{ik}}^2$ into ${\tilde{\beta}_{ik}}^1$, we consider Problem \ref{eq.1} in a chance-constrained uncertain-based framework.  In this framework, the problem is as follows. 
\begin{equation}\label{eq.2}
	\begin{aligned}
		&	\min \quad \mathbb{E}\bigg[\sum\limits_{i=1}^{N_0}{\tilde{\beta}_{i0}}^1\prod\limits_{j=1}^{n}x_j^{\alpha_{0ij}}\bigg]\\
		&	\text{s.t.}\\
		& \qquad \mathcal{M}\bigg(\sum\limits_{i=1}^{N_k}{\tilde{\beta}_{ik}}^1\prod\limits_{j=1}^{n}x_j^{\alpha_{kij}}\leq 1\bigg)\geq \gamma,k=1,2,\ldots K,
	\end{aligned}
\end{equation} 
where $\mathcal{M}$ is the single-fold uncertain measure, $\gamma\in (0,1)$ is a predefined confidence level. To derive the equivalent deterministic form of Problem \ref{eq.2}, we use the following theorem developed by Liu \cite{Liu 2015}.
\begin{theorem}\cite{Liu 2015}\label{thm.12}
	Assume the optimization problem in a chance-constrained uncertain based framework as
\begin{equation}\label{eq.3}
	\begin{aligned}
		&	\min \quad \mathbb{E}\bigg[\tilde{f}_0{(\textbf{x};\xi_1,\xi_2,\ldots,\xi_n)}\bigg]\\
		&	\text{s.t.}\\
		& \qquad \mathcal{M}\bigg(\tilde{f}_k{(\textbf{x};\xi_1,\xi_2,\ldots,\xi_n)}\leq 0\bigg)\geq \gamma,k=1,2,\ldots K,
	\end{aligned}
\end{equation}
where $\gamma\in (0,1)$ is a predefined confidence level. Let $\xi_1,\xi_2,\ldots,\xi_n$ be independent UVs having regular UDs $\Phi_{\xi_1},\Phi_{\xi_2},\ldots,\Phi_{\xi_n}$, respectively. If the objective and constraint functions are strictly increasing with respect to $\xi_1,\xi_2,\ldots,\xi_n$, then the deterministic form of Problem \ref{eq.3} is as follows.
\begin{equation}\label{eq.4}
	\begin{aligned}
		&	\min \quad \int\limits_{0}^{1}\tilde{f}_0{\big(\textbf{x};\Phi_{\xi_1}^{-1}(\gamma),\Phi_{\xi_2}^{-1}(\gamma),\ldots,\Phi_{\xi_n}^{-1}(\gamma)\big)}d\gamma\\
		&	\text{s.t.}\\
		& \qquad \tilde{f}_k{\big(\textbf{x};\Phi_{\xi_1}^{-1}(\gamma),\Phi_{\xi_2}^{-1}(\gamma),\ldots,\Phi_{\xi_n}^{-1}(\gamma)\big)}\leq 0,k=1,2,\ldots K.
	\end{aligned}
\end{equation}
\end{theorem}
\par Based on Theorem \ref{thm.12}, Problem \ref{eq.2} becomes in a deterministic form as
\begin{equation}\label{eq.5}
	\begin{aligned}
		&	\min \quad \sum\limits_{i=1}^{N_0}\beta_{i0}\prod\limits_{j=1}^{n}x_j^{\alpha_{0ij}}\\
		&	\text{s.t.}\\
		& \qquad \sum\limits_{i=1}^{N_k}\beta_{ik}\prod\limits_{j=1}^{n}x_j^{\alpha_{kij}}\leq 1,k=1,2,\ldots K,
	\end{aligned}
\end{equation} 
where $\gamma\in (0,1)$ is a predefined confidence level, $\beta_{i0}=\int\limits_{0}^{1}\Phi_{\tilde{\beta}_{i0}^1}^{-1}(\gamma)d\gamma,i=1,2,\ldots, N_0;$ $\beta_{ik}=\Phi_{\tilde{\beta}_{ik}^1}^{-1}(\gamma),i=1,2,\ldots,N_k;k=1,2,\ldots,K.$ Problem \ref{eq.5} is a deterministic primal GP problem, where $N_0$ and $N_k$ are the total number of terms of objective and $k^{th}$ constraint, respectively. To define the dual problem, let $N=\sum\limits_{k=0}^{K}N_k$ be the total numbers of terms presents in Problem \ref{eq.5}, and $\delta_{ik}$ be the dual variables such that $\lambda_k=\sum\limits_{i=1}^{N_k}{\delta_{ik}},k=0,1,2,\ldots,K.$ Then the dual problem is 
\begin{equation}\label{eq.6}
	\begin{aligned}
		&\max \quad V(\delta)=\prod_{i=1}^{N}\bigg(\frac{\beta_{ik}}{{\delta_{ik}}}\bigg)^{{\delta_{ik}}}\bigg(\lambda_k\bigg)^{\lambda_k}\\
		&\text{s.t.}\\
		& \qquad \sum\limits_{i=1}^{N_0}{\delta_{i0}}=\lambda_0=1, \quad (\text{Normality condition})\\
		&\qquad\sum\limits_{i=1}^{N}{\delta_{ik}}{\alpha_{kij}}=0,j=1,2,\ldots,n. \quad (\text{Orthogonality conditions})
	\end{aligned}
\end{equation}
\par In a GP problem, the dual is solved, and then using the primal-dual relationship, primal decision variables are found. Here is the relationship between primal and dual problems.\\\\
\textbf{Primal-dual Relationship:} 
Due to the strong duality theorem \cite{Duffin 1967,Duffin 1973}, here is how the primal and dual of a GP problem relate to each other. 
\begin{equation}\label{eq.7}
	\begin{aligned}
		&\sum\limits_{j=1}^n{\alpha_{0ij} } \ln(x_j)=\ln \bigg(\frac{{\delta_{i0}}f_0(\textbf{x}) }{\beta_{i0}}\bigg), i=1,2,\ldots,N_{0},\\
		&\sum\limits_{j=1}^n{\alpha_{kij} } \ln(x_j)=\ln \bigg(\frac{\delta_{ik}} { \lambda_k \beta_{ik}}\bigg),i=1,2,\ldots,N_{k},k=1,2,\ldots,K.
	\end{aligned}
\end{equation}
\begin{remark} 
	Depending on $N$, we have the following two cases.
	\begin{description}
		\item[I.] If $N \geq n+1$, then Problem \ref{eq.6} is feasible as the number of equations is less than or equal to the number of dual variables, which guarantees the existence of a solution to the dual problem \cite{Beightler 1976}.
		\item[II.] If $N < n+1$, then Problem \ref{eq.6} is inconsistent as the number of equations is greater than the number of dual variables. It guarantees that there is no analytical solution to the dual problem. However, an approximate solution can be found by the least square or linear programming method \cite{sinha 1987}.
	\end{description}
\end{remark}
\section{Numerical Example}\label{sec.5} In this section, we provide a numerical example of a GP problem with two-fold uncertainty. We consider the GP problem with two-fold uncertain coefficients as
\begin{equation}\label{eq.8}
	\begin{aligned}
		&	\min \quad \tilde{f}_0(\textbf{x})=\tilde{\beta}_{10}^2x_1 x_2+\tilde{\beta}_{20}^2x_2 x_3+
		\tilde{\beta}_{30}^2x_1 x_3\\
		&	\text{s.t.} \\
		&\qquad \tilde{f}_1(\textbf{x})= \frac{\tilde{\beta}_{11}^2}{x_1 x_2x_3}\leq 1,\\
		&\qquad	x_1,x_2,x_3>0.
	\end{aligned}
\end{equation}
where $\tilde{\beta}_{10}^2,\tilde{\beta}_{20}^2,\tilde{\beta}_{30}^2,\tilde{\beta}_{11}^2$ are independent two-fold UVs.
In this numerical example, we consider the following two cases.\\\\
\textbf{Case I:} In this case, we assume that the coefficients of the objective and constraints of Problem \ref{eq.8} are triangular two-fold UVs, which are given as $\tilde{\beta}_{10}^2=\mathcal{TRI}(10,20,25;0.5,0.6)$, $\tilde{\beta}_{20}^2=\mathcal{TRI}(30,40,50;0.4,0.6)$, $\tilde{\beta}_{30}^2=\mathcal{TRI}(15,25,30;0.4,0.5)$, $\tilde{\beta}_{11}^2=\mathcal{TRI}(6,8,9;0.5,0.7)$. For simplicity, we use expected value reduction method to reduce two-fold UVs $\tilde{\beta}_{10}^2,\tilde{\beta}_{20}^2,\tilde{\beta}_{30}^2,\tilde{\beta}_{11}^2$ into single-fold UVs $\tilde{\beta}_{10}^1,\tilde{\beta}_{20}^1,\tilde{\beta}_{30}^1,\tilde{\beta}_{11}^1$, respectively. Based on Theorem \ref{thm.8}, UDs of reduced single-fold UVs $\tilde{\beta}_{10}^1,\tilde{\beta}_{20}^1,\tilde{\beta}_{30}^1,\tilde{\beta}_{11}^1$ are as follows. 
\[\Phi_{\tilde{\beta}_{10}^1}(x)=\left\{
\begin{array}{ll}
	0, & x\leq 10;\\
	\frac{(x-10)^2}{150}+\frac{1}{20}\frac{(x-10)^2}{150}, & 10<x\leq 10+5\sqrt{2};\\
	\frac{(x-10)^2}{150}+\frac{1}{20}\big[\frac{2}{3}-\frac{(x-10)^2}{150}\big], & 10+5\sqrt{2}\leq x<20;\\
	\frac{2}{3}, & x=20;\\
	1-\frac{(25-x)^2}{75}+\frac{1}{20}\big[\frac{1}{3}-\frac{(25-x)^2}{75}\big], & 20<x\leq 25-\frac{5}{\sqrt{2}};\\
	1-\frac{(25-x)^2}{75}+\frac{1}{20}\frac{(25-x)^2}{75}, &  25-\frac{5}{\sqrt{2}}\leq x<25;\\
	1, & x\geq 25;\\
\end{array}
\right.
\]

\[\Phi_{\tilde{\beta}_{20}^1}(x)=\left\{
\begin{array}{ll}
	0, & x\leq 30;\\
	\frac{(x-30)^2}{200}+\frac{1}{10}\frac{(x-30)^2}{200}, & 30<x\leq 30+5\sqrt{2};\\
	\frac{(x-30)^2}{200}+\frac{1}{10}\big[\frac{1}{2}-\frac{(x-30)^2}{200}\big], & 30+5\sqrt{2}\leq x<40;\\
	\frac{1}{2}, & x=40;\\
	1-\frac{(50-x)^2}{200}+\frac{1}{10}\big[\frac{1}{2}-\frac{(50-x)^2}{200}\big], & 40<x\leq 50-5\sqrt{2};\\
	1-\frac{(50-x)^2}{200}+\frac{1}{10}\frac{(50-x)^2}{200}, &  50-5\sqrt{2}\leq x<50;\\
	1, & x\geq 50;\\
\end{array}
\right.
\]

\[\Phi_{\tilde{\beta}_{30}^1}(x)=\left\{
\begin{array}{ll}
	0, & x\leq 15;\\
	\frac{(x-15)^2}{150}+\frac{1}{20}\frac{(x-15)^2}{150}, & 15<x\leq 15+5\sqrt{2};\\
	\frac{(x-15)^2}{150}+\frac{1}{20}\big[\frac{2}{3}-\frac{(x-15)^2}{150}\big], & 15+5\sqrt{2}\leq x<25;\\
	\frac{2}{3}, & x=25;\\
	1-\frac{(30-x)^2}{75}+\frac{1}{20}\big[\frac{1}{3}-\frac{(30-x)^2}{75}\big], & 25<x\leq 30-\frac{5}{\sqrt{2}};\\
	1-\frac{(30-x)^2}{75}+\frac{1}{20}\frac{(30-x)^2}{75}, &  30-\frac{5}{\sqrt{2}}\leq x<30;\\
	1, & x\geq 30;\\
\end{array}
\right.
\]

\[\Phi_{\tilde{\beta}_{11}^1}(x)=\left\{
\begin{array}{ll}
	0, & x\leq 6;\\
	\frac{(x-6)^2}{6}+\frac{1}{10}\frac{(x-6)^2}{6}, & 6<x\leq 6+\sqrt{2};\\
	\frac{(x-6)^2}{6}+\frac{1}{10}\big[\frac{2}{3}-\frac{(x-6)^2}{6}\big], & 6+\sqrt{2}\leq x<8;\\
	\frac{2}{3}, & x=8;\\
	1-\frac{(9-x)^2}{3}+\frac{1}{10}\big[\frac{1}{3}-\frac{(9-x)^2}{3}\big], & 8<x\leq 9-\frac{1}{\sqrt{2}};\\
	1-\frac{(9-x)^2}{3}+\frac{1}{10}\frac{(9-x)^2}{3}, &  9-\frac{1}{\sqrt{2}}\leq x<9;\\
	1, & x\geq 9.\\
\end{array}
\right.
\]   
Therefore, the inverse of those UDs are as follows.
\[\Phi_{\tilde{\beta}_{10}^1}^{-1}(\gamma)=\left\{
\begin{array}{ll}
	10+\sqrt{\frac{1000}{7}\gamma}, & 0<\gamma\leq \frac{7}{20};\\
	10+\sqrt{\frac{3000}{19}\gamma-\frac{100}{19}}, & \frac{7}{20}\leq\gamma\leq \frac{2}{3};\\
	25-\sqrt{\frac{1525}{21}-\frac{500}{7}\gamma}, & \frac{2}{3}\leq\gamma\leq \frac{101}{120};\\
	25-\sqrt{\frac{1500}{19}-\frac{1500}{19}\gamma}, & \frac{101}{120}\leq\gamma<1;
\end{array}
\right.
\]

\[\Phi_{\tilde{\beta}_{20}^1}^{-1}(\gamma)=\left\{
\begin{array}{ll}
	30+\sqrt{\frac{2000}{11}\gamma}, & 0<\gamma\leq \frac{11}{40};\\
	30+\sqrt{\frac{2000}{9}\gamma-\frac{100}{9}}, & \frac{11}{40}\leq\gamma\leq \frac{1}{2};\\
	50-\sqrt{\frac{2100}{11}-\frac{2000}{11}\gamma}, & \frac{1}{2}\leq\gamma\leq \frac{31}{40};\\
	50-\sqrt{\frac{2000}{9}-\frac{2000}{9}\gamma}, & \frac{31}{40}\leq\gamma<1;
\end{array}
\right.
\]

\[\Phi_{\tilde{\beta}_{30}^1}^{-1}(\gamma)=\left\{
\begin{array}{ll}
	15+\sqrt{\frac{1000}{7}\gamma}, & 0<\gamma\leq \frac{7}{20};\\
	15+\sqrt{\frac{3000}{19}\gamma-\frac{100}{19}}, & \frac{7}{20}\leq\gamma\leq \frac{2}{3};\\
	30-\sqrt{\frac{1525}{21}-\frac{500}{7}\gamma}, & \frac{2}{3}\leq\gamma\leq \frac{101}{120};\\
	30-\sqrt{\frac{1500}{19}-\frac{1500}{19}\gamma}, & \frac{101}{120}\leq\gamma<1;
\end{array}
\right.
\]

\[\Phi_{\tilde{\beta}_{11}^1}^{-1}(\gamma)=\left\{
\begin{array}{ll}
	6+\sqrt{\frac{60}{11}\gamma}, & 0<\gamma\leq \frac{11}{30};\\
	6+\sqrt{\frac{20}{3}\gamma-\frac{4}{9}}, & \frac{11}{30}\leq\gamma\leq \frac{2}{3};\\
	9-\sqrt{\frac{31}{11}-\frac{30}{11}\gamma}, & \frac{2}{3}\leq\gamma\leq \frac{51}{60};\\
	9-\sqrt{\frac{10}{3}-\frac{10}{3}\gamma}, & \frac{51}{60}\leq\gamma<1.
\end{array}
\right.
\]  
Now, we have $\beta_{10}=\int\limits_{0}^{1}\Phi_{\tilde{\beta}_{10}^1}^{-1}(\gamma)d\gamma,\beta_{20}=\int\limits_{0}^{1}\Phi_{\tilde{\beta}_{20}^1}^{-1}(\gamma)d\gamma,\text{ and }\beta_{30}=\int\limits_{0}^{1}\Phi_{\tilde{\beta}_{30}^1}^{-1}(\gamma)d\gamma.$ After calculation, we have $\beta_{10}\approx 18.252,\beta_{20}\approx 39.804,\text{ and }\beta_{30}\approx 23.252.$ Therefore, the equivalent deterministic form of Problem \ref{eq.8} is
\begin{equation}\label{eq.9}
	\begin{aligned}
		&	\min \quad 18.252x_1 x_2+39.804x_2 x_3+
		23.252x_1 x_3\\
		&	\text{s.t.} \\
		&\qquad  \frac{\Phi_{\tilde{\beta}_{11}^1}^{-1}(\gamma)}{x_1 x_2x_3}\leq 1,\\
		&\qquad	x_1,x_2,x_3>0.
	\end{aligned}
\end{equation}
where $\Phi_{\tilde{\beta}_{11}^1}^{-1}(\gamma)$ is given above. Here is the dual of Problem \ref{eq.9} as follows.
\begin{equation}\label{eq.10}
	\begin{aligned}
		&\max \quad V(\delta)=\bigg(\frac{18.252}{{\delta_{1}}}\bigg)^{{\delta_{1}}}\bigg(\frac{39.804}{{\delta_{2}}}\bigg)^{{\delta_{2}}} \bigg(\frac{23.252}{{\delta_{3}}}\bigg)^{{\delta_{3}}}\bigg(\frac{\Phi_{\tilde{\beta}_{11}^1}^{-1}(\gamma)}{{\delta_{4}}}\bigg)^{{\delta_{4}}}\big(\delta_4\big)^{\delta_4}\\
		&\text{s.t.}\\
		& \qquad \delta_1+\delta_{2}+\delta_{3}=1, \quad (\text{Normality condition})\\
		&\qquad \delta_{1}+\delta_{3}-\delta_{4}=0,\\
		&\qquad \delta_{1}+\delta_{2}-\delta_{4}=0, \quad (\text{Orthogonality conditions})\\
		&\qquad \delta_{2}+\delta_{3}-\delta_{4}=0.
	\end{aligned}
\end{equation}
Solving Problem \ref{eq.10}, we get dual solution. Consequently, by the strong duality theorem, we obtain the primal decision variables and expected objective value. We find the solutions for confidence levels $\gamma \in (0,1)$, which are given in Table \ref{table.1}. \\
\begin{table}[htb]
	\caption{ Optimal solutions under triangular two-fold uncertainty}
	\label{table.1}
	\centering
	\begin{tabular}{c c c c c c c c c }
		\hline\hline
		$\gamma$ & $x_1^*$ & $x_2^*$ & $x_3^*$ & $\delta_1^*$ & $\delta_2^*$ & $\delta_3^*$ & $\delta_4^*$ & $\mathbb{E}[\tilde{f}_0(\textbf{x}^*)]$\\ [0.5ex]
		\hline\hline
		0.1 & 2.930 & 1.712 & 1.344 & 0.333 & 0.333 & 0.333 & 0.667 & 274.632 \\
	0.2 & 2.974 & 1.737 & 1.364 & 0.333 & 0.333 & 0.333 & 0.667 & 282.857\\
		0.3 & 3.006 & 1.756 & 1.379 & 0.333 & 0.333 & 0.333 & 0.667 & 289.113 \\
	0.4 & 3.035 & 1.773 & 1.392 & 0.333 & 0.333 & 0.333 & 0.667 & 294.700\\
	0.5 & 3.063 & 1.789 & 1.405 & 0.333 & 0.333 & 0.333 & 0.667 & 300.156\\
		0.6 & 3.088 & 1.804 & 1.416 & 0.333 & 0.333 & 0.333 & 0.667 & 304.971 \\
	0.7 & 3.109 & 1.816 & 1.425 & 0.333 & 0.333 & 0.333 & 0.667 & 309.107\\
	0.8 & 3.128 & 1.828 & 1.435 & 0.333 & 0.333 & 0.333 & 0.667 & 313.064 \\
		0.9 & 3.156 & 1.844 & 1.447 & 0.333 & 0.333 & 0.333 & 0.667 & 318.663 \\ [1ex]
		\hline

	\end{tabular}
	
\end{table}\\
Figure \ref{fig:11} shows the expected objective value with respect to confidence level under triangular two-fold uncertainty based on Table \ref{table.1}. It is observed that the expected objective value is increasing with respect to the confidence level. \\
\begin{figure}[htb] 
	\centering 
	\includegraphics[scale=0.225]{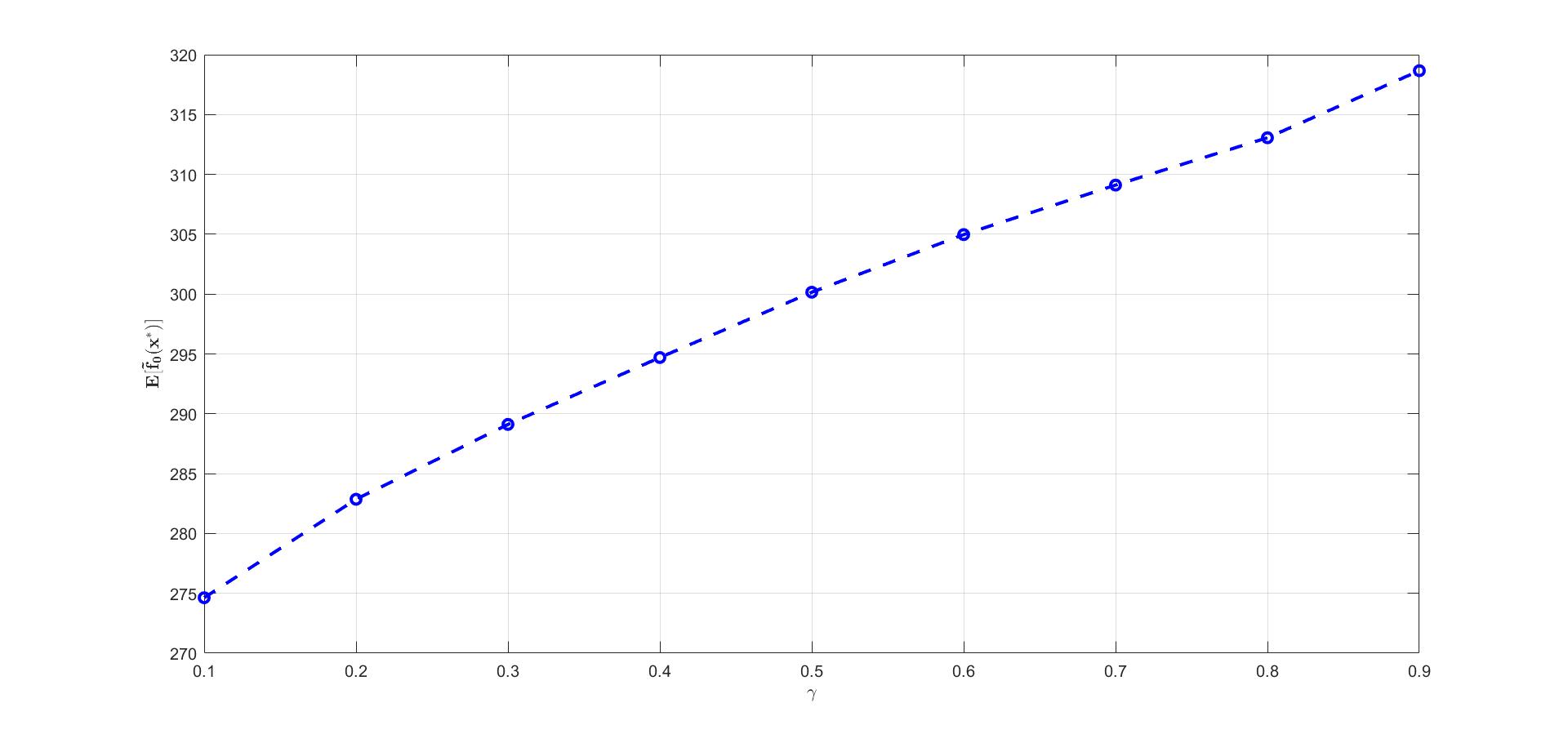} 
	\caption{Expected objective value with respect to confidence level under triangular two-fold uncertainty.}
	\label{fig:11}
\end{figure}
\\
\textbf{Case II:} In this case, we assume that the coefficients of the objective and constraints of Problem \ref{eq.8} are trapezoidal two-fold UVs, which are given as $\tilde{\beta}_{10}^2=\mathcal{TRA}(10,15,20,25;0.5,0.6)$, $\tilde{\beta}_{20}^2=\mathcal{TRA}(30,40,50,60;0.4,0.6)$, $\tilde{\beta}_{30}^2=\mathcal{TRA}(15,20,25,30;0.4,0.5)$, $\tilde{\beta}_{11}^2=\mathcal{TRA}(6,7,8,9;0.5,0.7)$. Similar to earlier case, we use expected value reduction method to reduce two-fold UVs $\tilde{\beta}_{10}^2,\tilde{\beta}_{20}^2,\tilde{\beta}_{30}^2,\tilde{\beta}_{11}^2$ into single-fold UVs $\tilde{\beta}_{10}^1,\tilde{\beta}_{20}^1,\tilde{\beta}_{30}^1,\tilde{\beta}_{11}^1$, respectively. Based on Theorem \ref{thm.11}, UDs of reduced single-fold UVs $\tilde{\beta}_{10}^1,\tilde{\beta}_{20}^1,\tilde{\beta}_{30}^1,\tilde{\beta}_{11}^1$ are as follows. 
\[\Phi_{\tilde{\beta}_{10}^1}(x)=\left\{
\begin{array}{ll}
	0, & x\leq 10;\\
	\frac{(x-10)^2}{100}+\frac{1}{20}\frac{(x-10)^2}{100}, & 10<x\leq 10+\frac{5}{\sqrt{2}};\\
	\frac{(x-10)^2}{100}+\frac{1}{20}\big[\frac{1}{4}-\frac{(x-10)^2}{100}\big], & 10+\frac{5}{\sqrt{2}}\leq x<15;\\
	\frac{1}{4}, & x=15;\\
	\frac{2x-25}{20}+\frac{1}{20}\big[\frac{2x-25}{20}-\frac{1}{4}\big], & 15< x\leq \frac{35}{2};\\
	\frac{2x-25}{20}+\frac{1}{20}\big[\frac{3}{4}-\frac{2x-25}{20}\big], & \frac{35}{2}\leq x<20;\\
	\frac{3}{4}, & x=20;\\
	1-\frac{(25-x)^2}{100}+\frac{1}{20}\big[\frac{1}{4}-\frac{(25-x)^2}{100}\big], & 20<x\leq 25-\frac{5}{\sqrt{2}};\\
	1-\frac{(25-x)^2}{100}+\frac{1}{20}\frac{(25-x)^2}{100}, &  25-\frac{5}{\sqrt{2}}\leq x<25;\\
	1, & x\geq 25;\\
\end{array}
\right.
\]

\[\Phi_{\tilde{\beta}_{20}^1}(x)=\left\{
\begin{array}{ll}
	0, & x\leq 30;\\
	\frac{(x-30)^2}{400}+\frac{1}{10}\frac{(x-30)^2}{400}, & 30<x\leq 30+5\sqrt{2};\\
	\frac{(x-30)^2}{400}+\frac{1}{10}\big[\frac{1}{4}-\frac{(x-30)^2}{400}\big], & 30+5\sqrt{2}\leq x<40;\\
	\frac{1}{4}, & x=40;\\
	\frac{2x-70}{40}+\frac{1}{10}\big[\frac{2x-70}{40}-\frac{1}{4}\big], & 40< x\leq 45;\\
	\frac{2x-70}{40}+\frac{1}{10}\big[\frac{3}{4}-\frac{2x-70}{40}\big], & 45\leq x<50;\\
	\frac{3}{4}, & x=50;\\
	1-\frac{(60-x)^2}{400}+\frac{1}{10}\big[\frac{1}{4}-\frac{(60-x)^2}{400}\big], & 50<x\leq 60-5\sqrt{2};\\
	1-\frac{(60-x)^2}{400}+\frac{1}{10}\frac{(60-x)^2}{400}, &  60-5\sqrt{2}\leq x<60;\\
	1, & x\geq 60;\\
\end{array}
\right.
\]

\[\Phi_{\tilde{\beta}_{30}^1}(x)=\left\{
\begin{array}{ll}
	0, & x\leq 15;\\
	\frac{(x-15)^2}{100}+\frac{1}{20}\frac{(x-15)^2}{100}, & 15<x\leq 15+\frac{5}{\sqrt{2}};\\
	\frac{(x-15)^2}{100}+\frac{1}{20}\big[\frac{1}{4}-\frac{(x-15)^2}{100}\big], & 15+\frac{5}{\sqrt{2}}\leq x<20;\\
	\frac{1}{4}, & x=20;\\
	\frac{2x-35}{20}+\frac{1}{20}\big[\frac{2x-35}{20}-\frac{1}{4}\big], & 20< x\leq \frac{45}{2};\\
	\frac{2x-35}{20}+\frac{1}{20}\big[\frac{3}{4}-\frac{2x-35}{20}\big], & \frac{45}{2}\leq x<25;\\
	\frac{3}{4}, & x=25;\\
	1-\frac{(30-x)^2}{100}+\frac{1}{20}\big[\frac{1}{4}-\frac{(30-x)^2}{100}\big], & 25<x\leq 30-\frac{5}{\sqrt{2}};\\
	1-\frac{(30-x)^2}{100}+\frac{1}{20}\frac{(30-x)^2}{100}, &  30-\frac{5}{\sqrt{2}}\leq x<30;\\
	1, & x\geq 30;\\
\end{array}
\right.
\]

\[\Phi_{\tilde{\beta}_{11}^1}(x)=\left\{
\begin{array}{ll}
	0, & x\leq 6;\\
	\frac{(x-6)^2}{4}+\frac{1}{10}\frac{(x-6)^2}{4}, & 6<x\leq 6+\frac{1}{\sqrt{2}};\\
	\frac{(x-6)^2}{4}+\frac{1}{10}\big[\frac{1}{4}-\frac{(x-6)^2}{4}\big], & 6+\frac{1}{\sqrt{2}}\leq x<7;\\
	\frac{1}{4}, & x=7;\\
	\frac{2x-13}{4}+\frac{1}{10}\big[\frac{2x-13}{4}-\frac{1}{4}\big], & 7< x\leq \frac{15}{2};\\
	\frac{2x-13}{4}+\frac{1}{10}\big[\frac{3}{4}-\frac{2x-13}{4}\big], & \frac{15}{2}\leq x<8;\\
	\frac{3}{4}, & x=8;\\
	1-\frac{(9-x)^2}{4}+\frac{1}{10}\big[\frac{1}{4}-\frac{(9-x)^2}{4}\big], & 8<x\leq 9-\frac{1}{\sqrt{2}};\\
	1-\frac{(9-x)^2}{4}+\frac{1}{10}\frac{(9-x)^2}{4}, &  9-\frac{1}{\sqrt{2}}\leq x<9;\\
	1, & x\geq 9.\\
\end{array}
\right.
\]   
Therefore, the inverse of those UDs are as follows.
\[\Phi_{\tilde{\beta}_{10}^1}^{-1}(\gamma)=\left\{
\begin{array}{ll}
	10+\sqrt{\frac{2000}{21}\gamma}, & 0<\gamma\leq \frac{21}{160};\\
	10+\sqrt{\frac{2000}{19}\gamma-\frac{25}{19}}, & \frac{21}{160}\leq\gamma\leq \frac{1}{4};\\
	\frac{200}{21}\gamma+\frac{265}{21}, &\frac{1}{4}\leq \gamma\leq \frac{41}{80};\\
	\frac{200}{19}\gamma+\frac{230}{19}, &\frac{41}{80}\leq \gamma\leq \frac{3}{4};\\
	25-\sqrt{\frac{675}{7}-\frac{200}{21}\gamma}, & \frac{3}{4}\leq\gamma\leq \frac{141}{160};\\
	25-\sqrt{\frac{2000}{19}-\frac{2000}{19}\gamma}, & \frac{141}{160}\leq\gamma<1;
\end{array}
\right.
\]

\[\Phi_{\tilde{\beta}_{20}^1}^{-1}(\gamma)=\left\{
\begin{array}{ll}
30+\sqrt{\frac{4000}{11}\gamma}, & 0<\gamma\leq \frac{11}{80};\\
30+\sqrt{\frac{4000}{9}\gamma-\frac{100}{9}}, & \frac{11}{80}\leq\gamma\leq \frac{1}{4};\\
\frac{200}{11}\gamma+\frac{390}{11}, &\frac{1}{4}\leq \gamma\leq \frac{21}{40};\\
\frac{200}{9}\gamma+\frac{100}{3}, &\frac{21}{40}\leq \gamma\leq \frac{3}{4};\\
60-\sqrt{\frac{4100}{11}-\frac{4000}{11}\gamma}, & \frac{3}{4}\leq\gamma\leq \frac{71}{80};\\
60-\sqrt{\frac{4000}{9}-\frac{4000}{9}\gamma}, & \frac{71}{80}\leq\gamma<1;
\end{array}
\right.
\]

\[\Phi_{\tilde{\beta}_{30}^1}^{-1}(\gamma)=\left\{
\begin{array}{ll}
15+\sqrt{\frac{2000}{21}\gamma}, & 0<\gamma\leq \frac{21}{160};\\
15+\sqrt{\frac{2000}{19}\gamma-\frac{25}{19}}, & \frac{21}{160}\leq\gamma\leq \frac{1}{4};\\
\frac{200}{21}\gamma+\frac{370}{21}, &\frac{1}{4}\leq \gamma\leq \frac{41}{80};\\
\frac{200}{19}\gamma+\frac{325}{19}, &\frac{41}{80}\leq \gamma\leq \frac{3}{4};\\
30-\sqrt{\frac{675}{7}-\frac{200}{21}\gamma}, & \frac{3}{4}\leq\gamma\leq \frac{141}{160};\\
30-\sqrt{\frac{2000}{19}-\frac{2000}{19}\gamma}, & \frac{141}{160}\leq\gamma<1;
\end{array}
\right.
\]

\[\Phi_{\tilde{\beta}_{11}^1}^{-1}(\gamma)=\left\{
\begin{array}{ll}
6+\sqrt{\frac{40}{11}\gamma}, & 0<\gamma\leq \frac{11}{80};\\
6+\sqrt{\frac{40}{9}\gamma-\frac{1}{9}}, & \frac{11}{80}\leq\gamma\leq \frac{1}{4};\\
\frac{20}{11}\gamma+\frac{72}{11}, &\frac{1}{4}\leq \gamma\leq \frac{21}{40};\\
\frac{20}{9}\gamma+\frac{19}{3}, &\frac{21}{40}\leq \gamma\leq \frac{3}{4};\\
9-\sqrt{\frac{41}{11}-\frac{40}{11}\gamma}, & \frac{3}{4}\leq\gamma\leq \frac{71}{80};\\
9-\sqrt{\frac{40}{9}-\frac{40}{9}\gamma}, & \frac{71}{80}\leq\gamma<1.
\end{array}
\right.
\] 
Similar to earlier case, we have $\beta_{10}=\int\limits_{0}^{1}\Phi_{\tilde{\beta}_{10}^1}^{-1}(\gamma)d\gamma,\beta_{20}=\int\limits_{0}^{1}\Phi_{\tilde{\beta}_{20}^1}^{-1}(\gamma)d\gamma,\text{ and }\beta_{30}=\int\limits_{0}^{1}\Phi_{\tilde{\beta}_{30}^1}^{-1}(\gamma)d\gamma.$ After calculation, we have $\beta_{10}\approx 17.745,\beta_{20}\approx 44.779,\text{ and }\beta_{30}\approx 21.775.$ Therefore, the equivalent deterministic form of Problem \ref{eq.8} is
\begin{equation}\label{eq.11}
	\begin{aligned}
		&	\min \quad 17.745x_1 x_2+44.779x_2 x_3+
		21.775x_1 x_3\\
		&	\text{s.t.} \\
		&\qquad  \frac{\Phi_{\tilde{\beta}_{11}^1}^{-1}(\gamma)}{x_1 x_2x_3}\leq 1,\\
		&\qquad	x_1,x_2,x_3>0.
	\end{aligned}
\end{equation}
where $\Phi_{\tilde{\beta}_{11}^1}^{-1}(\gamma)$ is given above. Here is the dual of Problem \ref{eq.11} as follows.
\begin{equation}\label{eq.12}
	\begin{aligned}
		&\max \quad V(\delta)=\bigg(\frac{17.745}{{\delta_{1}}}\bigg)^{{\delta_{1}}}\bigg(\frac{44.779}{{\delta_{2}}}\bigg)^{{\delta_{2}}} \bigg(\frac{21.775}{{\delta_{3}}}\bigg)^{{\delta_{3}}}\bigg(\frac{\Phi_{\tilde{\beta}_{11}^1}^{-1}(\gamma)}{{\delta_{4}}}\bigg)^{{\delta_{4}}}\big(\delta_4\big)^{\delta_4}\\
		&\text{s.t. the same normality and orthogonality conditions given in Problem \ref{eq.10}}.
	\end{aligned}
\end{equation}
Similarly, solving Problem \ref{eq.12}, we get dual solution. Consequently, by the strong duality theorem, we obtain the primal decision variables and expected objective value. We find the solutions for confidence levels $\gamma \in (0,1)$, which are given in Table \ref{table.2}. \\ 
\begin{table}[htb]
\caption{ Optimal solutions under trapezoidal two-fold uncertainty}
\label{table.2}
\centering
\begin{tabular}{c c c c c c c c c }
	\hline\hline
	$\gamma$ & $x_1^*$ & $x_2^*$ & $x_3^*$ & $\delta_1^*$ & $\delta_2^*$ & $\delta_3^*$ & $\delta_4^*$ & $\mathbb{E}[\tilde{f}_0(\textbf{x}^*)]$\\ [0.5ex]
	\hline\hline
	0.1 & 3.248 & 1.579 & 1.287 & 0.333 & 0.333 & 0.333 & 0.667 & 273.098 \\
0.2 & 3.293 & 1.601 & 1.305 & 0.333 & 0.333 & 0.333 & 0.667 & 280.738\\
0.3 & 3.326 & 1.617 & 1.318 & 0.333 & 0.333 & 0.333 & 0.667 & 286.393 \\
0.4 & 3.354 & 1.631 & 1.329 & 0.333 & 0.333 & 0.333 & 0.667 & 291.273\\
	0.5 & 3.382 & 1.645 & 1.340 & 0.333 & 0.333 & 0.333 & 0.667 & 296.112\\
	0.6 & 3.414 & 1.660 & 1.353 & 0.333 & 0.333 & 0.333 & 0.667 & 301.699 \\
	0.7 & 3.447 & 1.676 & 1.366 & 0.333 & 0.333 & 0.333 & 0.667 & 307.496\\
	0.8 & 3.476 & 1.690 & 1.378 & 0.333 & 0.333 & 0.333 & 0.667 & 312.825 \\
	0.9 & 3.510 & 1.707 & 1.391 & 0.333 & 0.333 & 0.333 & 0.667 & 318.927 \\ [1ex]
	\hline

	\end{tabular}
	
\end{table}
\\
In Figure \ref{fig:12}, we show the expected objective value with respect to confidence level under trapezoidal two-fold uncertainty based on Table \ref{table.2}. Here, too, the expected objective value increases with confidence level.\\
\begin{figure}[htb] 
	\centering 
	\includegraphics[scale=0.225]{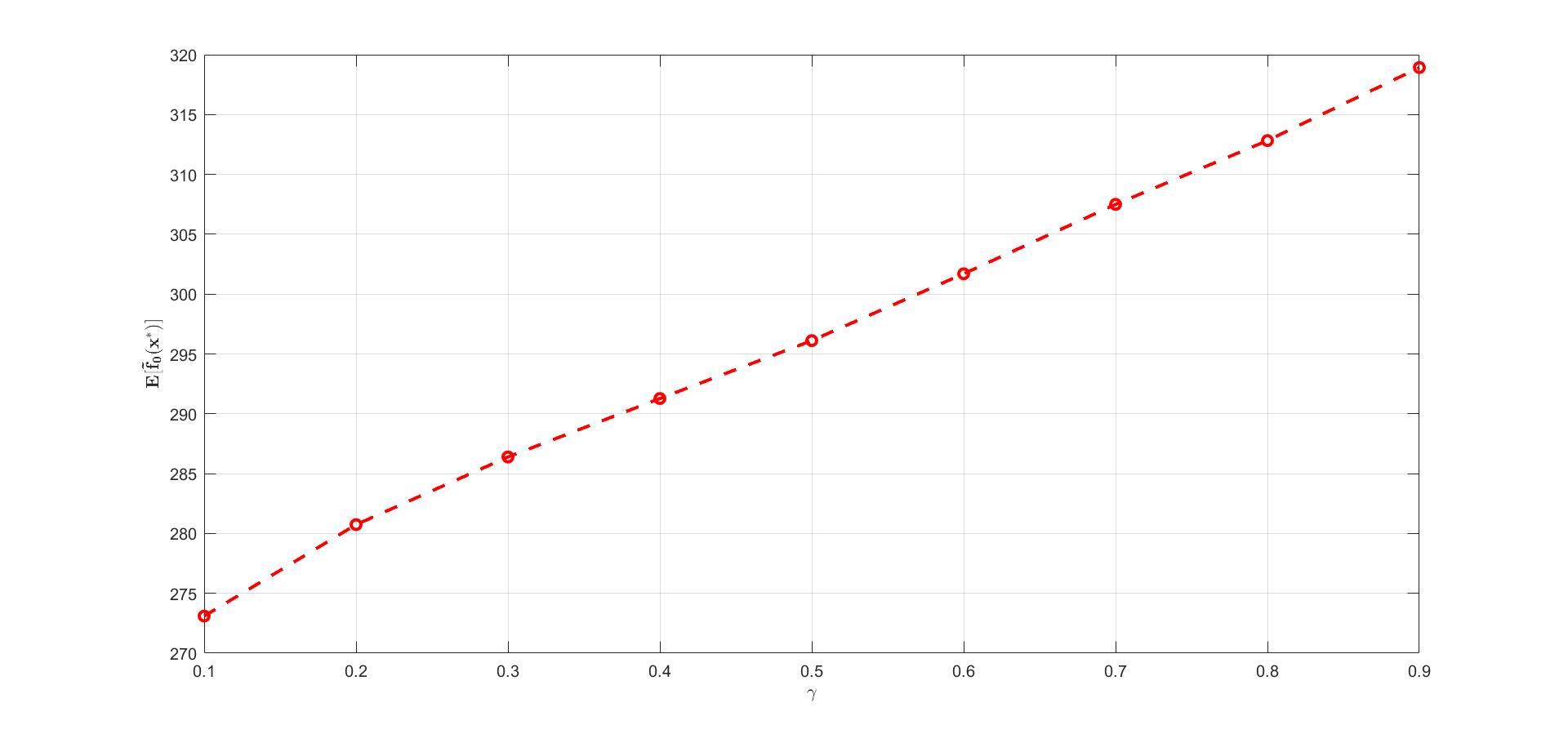} 
	\caption{Expected objective value with respect to confidence level under trapezoidal two-fold uncertainty.}
	\label{fig:12}
\end{figure}
\\
\section{Conclusion}\label{sec.6}
GP is a powerful technique for solving nonlinear optimization problems. Conventional GP problems assume that the parameters are precise and exact. However, in the real-world GP problem, the parameters may be imprecise and ambiguous. In this paper, we study the GP problem with the coefficients as UVs in an uncertain environment. We propose the definitions of triangular and trapezoidal two-fold UVs. We develop three reduction methods: optimistic, pessimistic, and expected value criteria, to reduce triangular and trapezoidal two-fold UVs into single-fold UVs. We apply those reduction methods to a GP problem with triangular and trapezoidal two-fold uncertainty. We show how the GP problem with two-fold uncertainty is transformed into its equivalent deterministic form. Finally, we provide the solution to a numerical example of a GP problem under triangular and trapezoidal two-fold uncertainty to show the efficiency and effectiveness of the procedure. In the numerical example, we use only the expected value reduction method to reduce two-fold uncertainty into single-fold uncertainty. One may also use optimistic and pessimistic reduction methods.
\par For application purposes, one may consider solid transportation problems and the inventory model under triangular and trapezoidal two-fold uncertainty for further research. \\\\
{\small \textbf{Acknowledgement }The first author is thankful to CSIR for financial support of this work through file No: 09\textbackslash 1059(0027)\textbackslash 2019-EMR-I.}
\\
\textbf{Data availability } Data sharing not applicable to this article as no datasets were generated. A random data set is taken for the example and the inventory model given in this paper.\\
\textbf{Conflict of interests} The authors have no relevant financial or non financial interests to disclose. The authors
declare that they have no conflict of interests.

\end{document}